\documentclass[a4paper]{article}

\usepackage{geometry}
\newcommand{\mypar}[1]{\vspace{.5em}\noindent\textbf{#1}}

\usepackage{amsthm}
\newtheorem{proposition}{Proposition}

\usepackage[utf8]{inputenc}
\usepackage{graphicx}

\usepackage{caption}
\usepackage{booktabs}
\usepackage[font=small]{subcaption}
\usepackage{amsmath}
\usepackage{amssymb}

\newcommand{\R}{\mathbb{R}}
\newcommand{\Rc}{\overline{\R}}

\newcommand{\bfln}{\begin{equation}\begin{aligned}}
\newcommand{\efln}{\end{aligned}\end{equation}}

\newcommand{\bfl}{\begin{equation}\begin{aligned}}
\newcommand{\efl}{\end{aligned}\end{equation}}

\newcommand{\dx}{\,\text{d}x}

\newcommand{\into}{\int_{\Omega}}

\newcommand{\bs}{\boldsymbol}
\newcommand{\BV}{\text{BV}}
\newcommand{\TV}{\text{TV}}

\DeclareMathOperator{\sgn}{sgn}

\begin{document}

\title{Inverse Scale Space Iterations for Non-Convex Variational Problems Using Functional Lifting}

\author{Danielle Bednarski and Jan Lellmann \\ \\ \small{Institute of Mathematics and Image Computing, University of L\"ubeck, Germany} \\  \small{ \texttt{{\{bednarski,lellmann\}@mic.uni-luebeck.de}}} }
\date{}

\maketitle
 
\renewcommand{\abstractname}{\vspace{-\baselineskip}}

\begin{abstract}
\noindent \textbf{Abstract.} Non-linear filtering approaches allow to obtain decompositions of images with respect to a non-classical notion of scale. The associated inverse scale space flow  can be obtained using the classical Bregman iteration  applied to a convex, absolutely one-homogeneous regularizer.
In order to extend these approaches to general energies with non-convex data term, we apply the Bregman iteration to a lifted version of the functional with sublabel-accurate discretization. We provide a condition for the subgradients of the regularizer under which this lifted iteration reduces to the standard Bregman iteration. We show experimental results for the convex and non-convex case.
\end{abstract}

\section{Motivation and Introduction}

We consider variational image processing problems with energies of the form
\bfln
F(u) := \underbrace{\into \rho(x,u(x)) \dx}_{H(u)} + \underbrace{\into \eta(\nabla u(x)) \dx}_{J(u)} ,
\label{eq:problem}
\efln
where the integrand $\eta : \R^d \mapsto \R$ of the \emph{regularizer} is non-negative and convex, and the integrand $\rho : \Omega \times \Gamma \mapsto \Rc$ of the \emph{data term} $H$ is proper, non-negative and possibly non-convex with respect to~$u$. We assume that the domain $\Omega \subset \R^d$ is open and bounded and that the range, or \emph{label space}, $\Gamma \subset \R$ is compact. 

Such problems are common in image reconstruction, segmentation, and motion estimation  \cite{book_aubert,book_scherzer}. We are mainly concerned with three distinct problem classes.  Whenever we are working with the \emph{total variation} regularizer, we use the abbreviation TV-\eqref{eq:problem}. If the data term is furthermore given by
\begin{equation}
\rho(x,u(x)) = \frac{\lambda}{2}(u(x)-f(x))^2
\label{eq:rof_data}
\end{equation}
for some input $f$ and $\lambda > 0$ we use the abbreviation ROF-\eqref{eq:problem}. For  data term~\eqref{eq:rof_data} and  arbitrary convex, absolute one-homogeneous regularizer~$\eta$ we write OH-\eqref{eq:problem}.

Consider the so-called \emph{inverse scale space flow} (ISS) \cite{osher2005iterative,burger2006nonlinear,burger2016spectral} equation
\begin{equation} \partial_s p(s) =f -u(s,\cdot), \quad p(s) \in \partial J(u(s,\cdot)), \quad p(0)=0,
\end{equation}
where $J$ is assumed to be convex and absolutely one-homogeneous. The evolution $u:[0,T]\times \Omega \to \R$ starts at $u(0,\cdot)=\text{mean}(f)$ and $p(s)$ is forced to lie in the subdifferential of the term $J$. E.g., for total variation regularization $J=\TV$, the flow $u(s,\cdot)$ progressively incorporates details of finer scales contained in the input image  $f$ as $s$ increases; for $s\to\infty$ the flow converges to the input image. 

By considering the derivative $u_s$, one can even define a non-linear decomposition of the input $f$ \cite{burger2015spectral,gilboa2016nonlinear} based on the solution $u$ of the inverse scale space flow and derive non-linear filters. Similar ideas have been developed for variational models of the form OH-\eqref{eq:problem} and \textit{gradient flow} formulations   \cite{burger2006nonlinear,benning2012ground,gilboa2013spectral,%gilboa2014total,
burger2016spectral,gilboa2017semi}.

\begin{figure}[t]
    \begin{subfigure}{.27\linewidth}
        \centering
        \includegraphics[width = \linewidth]{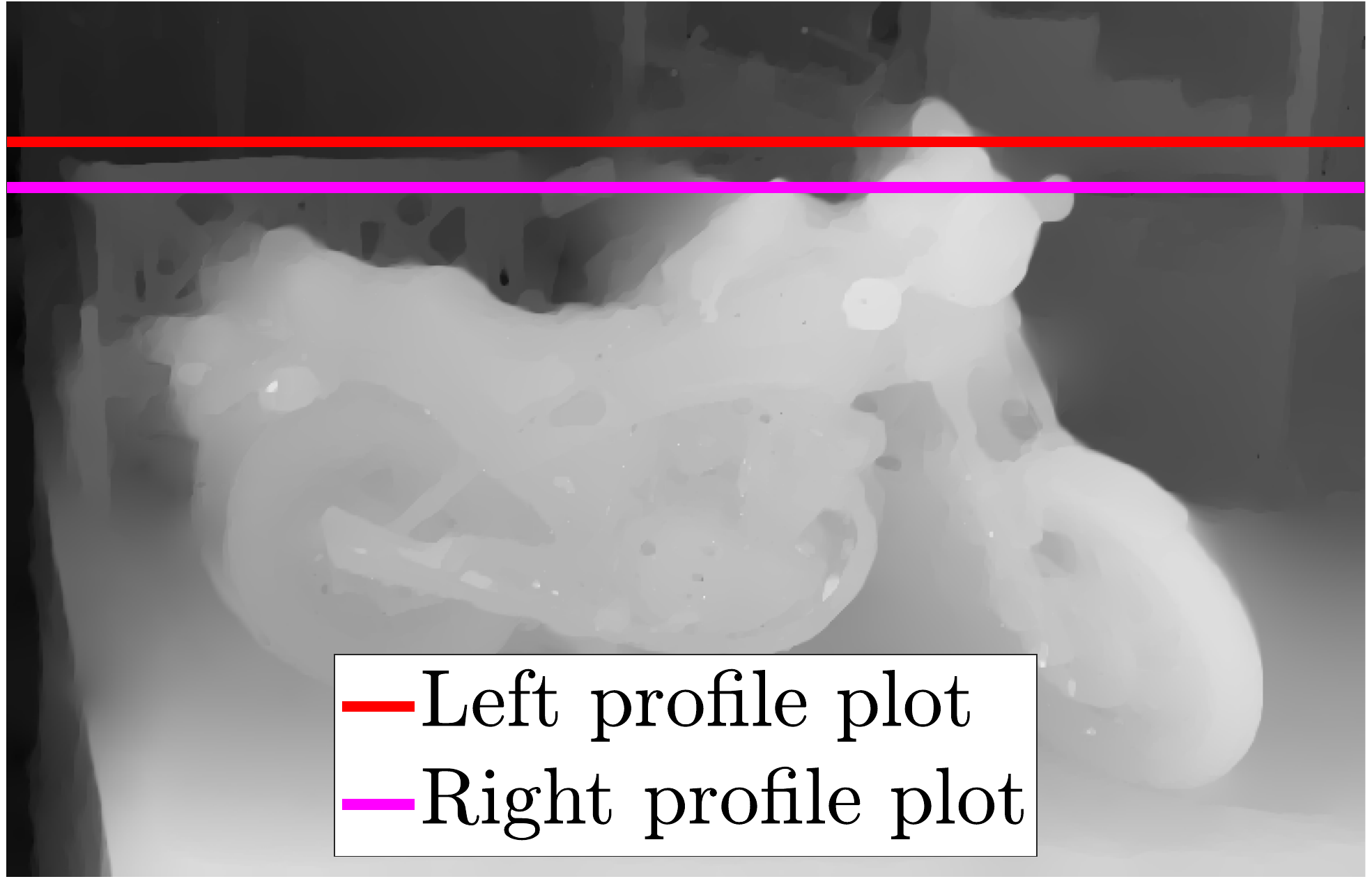}
    \end{subfigure}%
    \hspace{0.4cm}% Space between image A and B
    \begin{subfigure}{.30\linewidth}
        \centering
        \includegraphics[width = \linewidth]{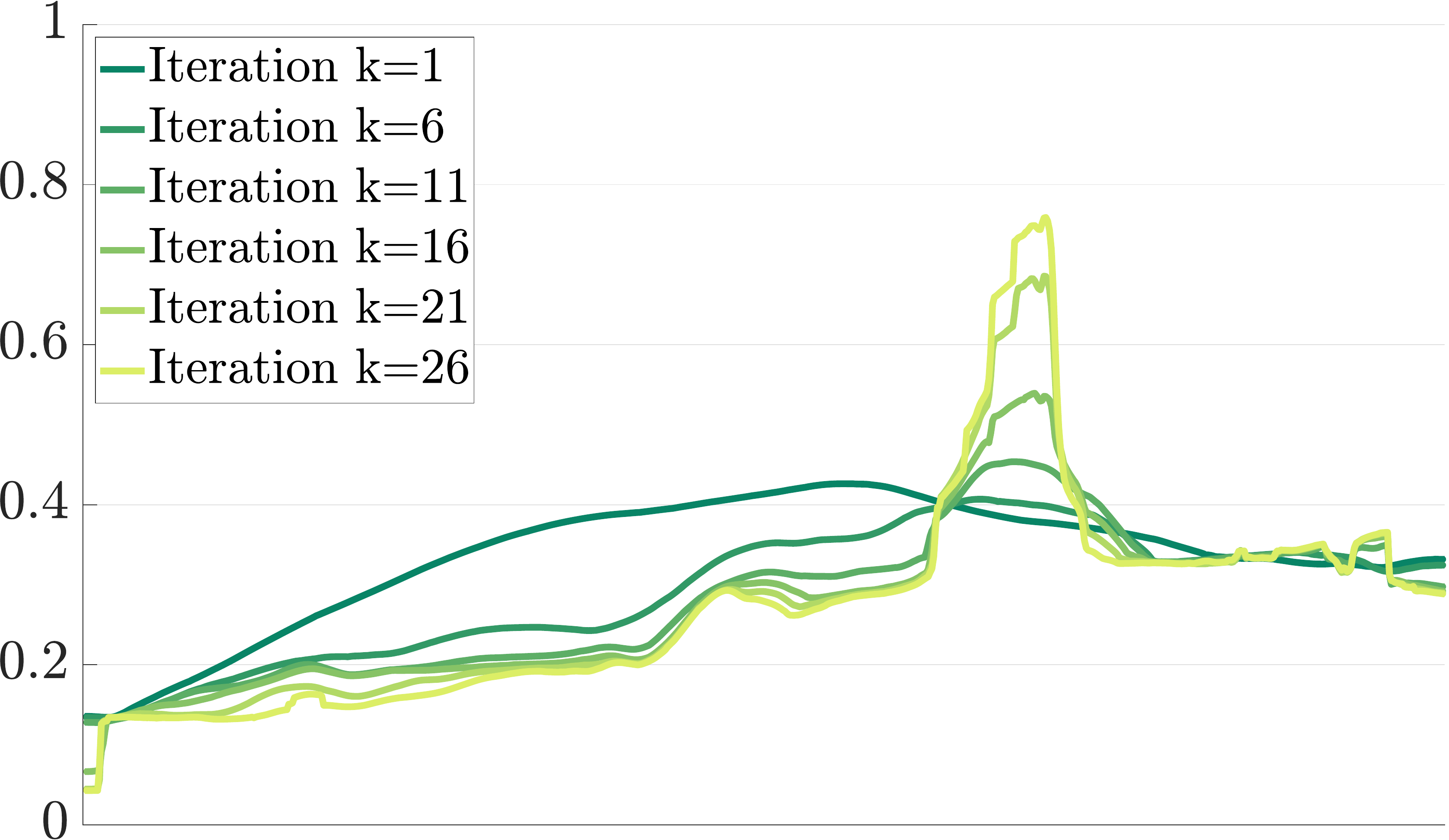}
    \end{subfigure}%
    \hspace{0.4cm}% Space between image A and B
    \begin{subfigure}{.30\linewidth}
        \centering
        \includegraphics[width = \linewidth]{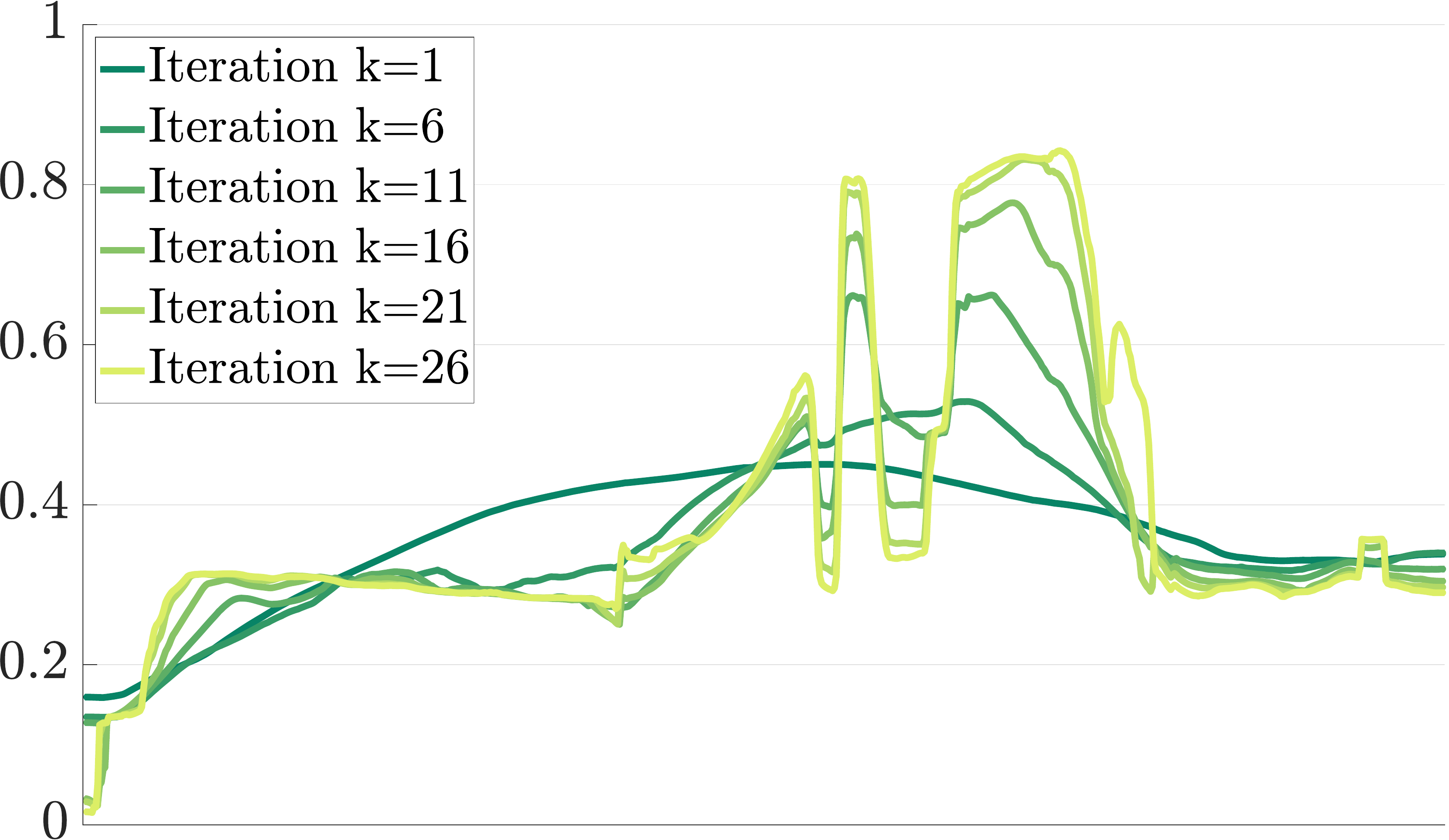}
        \end{subfigure}%        
     \caption{\textbf{Scale-space of solutions for non-convex depth estimation.} %Stereo matching is an established variational approach used for depth estimation.
Applying the sublabel-accurate lifting approach \cite{sublabel_cvpr} to the non-convex problem of depth estimation results in  a convex problem to which the  Bregman iteration~\cite{osher2005iterative} can be applied. In addition to the final depth map \textbf{(left)}, the Bregman iteration generates a scale space of solutions with increasing spatial detail, as can be seen from the two horizontal sections \textbf{(center, right).} %increasingly detailed estimated  allows to generate a \textbackslash textbThis allows to , resulting in a depth map . \textbf{Left:} Estimated depth map after 40 Bregman iterations. The two lines mark horizontal cuts for which we depict the profiles throughout the Bregman iteration. The profiles of the first iterations are merely a rough approximations of the actual depth while the following profiles comprise more and more details. This reminds of the inverse scale space flow. 
%
%\newline TV-\eqref{eq:problem} problem with data term  \eqref{eq:stereo_matching_data_term}. Solved with \texttt{prost} library \cite{sublabel_cvpr} on Middlebury benchmark dataset \cite{data_middlebury_bike}.
}
        \label{fig:stereo_profile}
\end{figure}

For problems in the class OH-\eqref{eq:problem}, the inverse scale space flow can be understood \cite{burger2006nonlinear} as a continuous limit of the so-called \emph{Bregman iteration} \cite{osher2005iterative}. For both the data term $H$ and regularizer $J$ being non-negative and convex (!) the Bregman iteration is defined as: 

 \noindent              \fbox{%
                \begin{minipage}[t]{0.95\linewidth}
                \textbf{Algorithm 1: Bregman iteration}
                        \begin{itemize}
                        \item[] Initialize $p_0 = 0$ and repeat for $k = 1,2,...$
                                                \begin{flalign}
                                u_k & \in \arg\min_{u} \{ H(u) + J(u) - \langle p_{k-1}, u \rangle \} , \label{eq:breg_uk}\\
                                p_{k} & \in \partial J(u_k) \label{eq:breg_qk}. 
                        \end{flalign}                   
                        \end{itemize}                                           
                \end{minipage}} \\

In case of the ROF-\eqref{eq:problem} problem the subgradient $p_k$ can be chosen explicitly as $p_k = p_{k-1} - \lambda (u_k - f)$. 
Further extensions include the \emph{split Bregman method} for $\ell_1$-regularized problems \cite{goldstein2009split} and the \emph{linearized Bregman iteration} for compressive sensing and sparse denoising \cite{cai2009linearized,osher2011fast}.

However, applying the Bregman iteration to variational problems with non-convex data term $H$ is not trivial since the well-definedness of the iterations as well as the convergence results in \cite{osher2005iterative} rely on the convexity of the data term. 
In \cite{hoeltgen2015bregman}, the Bregman iteration was used to solve a non-convex \emph{optical flow} problem, however, the approach relies on an iterative reduction to a convex problem using first-order Taylor approximations. 

In this work, we aim to apply the Bregman iteration to energies with a non-convex data term such as the non-convex stereo matching problem (Fig.~\ref{fig:stereo_profile} and Fig. \ref{fig:results2}). In order to do so, we follow a \emph{lifting approach:} Instead of minimizing the non-convex problem
\bfln
\inf_{u\in U} \{ H(u)+J(u) \}\label{eq:prob-ulf}
\efln
over some suitable (discrete or function) space $U$ we solve a \emph{lifted problem}
\bfln
\inf_{\bs{u}\in \bs{U}} \{ \bs{H}(\bs{u})+\bs{J}(\bs{u}) \}\label{eq:prob-lf}
\efln
over a larger space $\bs{U}$ but with \emph{convex} energies $\bs{H}, \bs{J}$. The Bregman iteration can then be performed on the convex problem \eqref{eq:prob-lf}:

 \noindent              \fbox{%
                \begin{minipage}[t]{0.95\linewidth}
                \textbf{Algorithm 2: Lifted Bregman iteration}
                        \begin{itemize}
                        \item[] Initialize $\bs{p}_0 = 0$ and repeat for $k = 1,2,...$
                                                \begin{flalign}
                                \bs{u}_k & \in \arg\min_{\bs{u}\in\bs{U}} \{ \bs{H}(\bs{u}) + \bs{J}(\bs{u}) - \langle \bs{p}_{k-1}, \bs{u} \rangle \} , \label{eq:breg_uk_lifted}\\
                                \bs{p}_{k} & \in \partial \bs{J}(\bs{u}_k) \label{eq:breg_qk_lifted}. 
                        \end{flalign}                   
                        \end{itemize}                                           
                \end{minipage}} \\

This allows to extend the Bregman iteration to non-convex data terms. Of course it raises the question whether the iterates of Alg.~1 and Alg.~2 are related, and whether the lifted method still generates a scale space in practice. In the following, we will investigate these questions.

\mypar{Outline and Contribution.}
In section 2 we summarize the sublabel-accurate relaxation approach for problems of the form TV-\eqref{eq:problem} as presented in \cite{sublabel_cvpr}. In section 3 we derive conditions under which the original and lifted Bregman iteration are equivalent. These conditions are in particular met by the anisotropic TV. In section 4 we validate these findings experimentally by comparing the original and lifted iteration on the convex ROF-\eqref{eq:problem} problem and present first numerical results on the non-convex stereo matching problem.

\mypar{Related Work.} %
In a fully discrete setting with discretized domain and finite range $\Gamma$,
Ishikawa and Geiger proposed  first lifting strategies for the labeling problem \cite{ishikawa_geiger,ishikawa}. Later the relaxation of the labeling problem was studied in a spatially continuous setting with binary \cite{chan_vese,chan_relax} and multiple labels \cite{depth2,lifting_continuous_multiclass}. 

Our work is based on methods for scalar but continuous range $\Gamma$ with first-order regularization in the spatially continuous setting \cite{lifting_tv,lifting_global_solutions}: The feasible set of scalar-valued functions $u:\Omega\to\Gamma$ is embedded into the convex set of functions  $v:\Omega\times\Gamma\to[0,1]$ by associating each function $u$ with the characteristic function of the subgraph, 
i.e., $\bs{1}_u(x,z) := 1$ if $u(x) > z$ and $0$ otherwise.
To extend the energy $F$ in \eqref{eq:problem} for $\Gamma = \R$ onto this larger space, a lifted convex functional $\mathcal{F}$ is defined:
\bfln
\mathcal{F}(v) := \sup_{\phi \in \mathcal{K}} \int_{\Omega \times \Gamma} \langle \phi, Dv \rangle ,
\label{eq:lifted_F}
\efln
where $Dv$ denotes the distributional derivative of $v$. With $\eta^*$ denoting the pointwise conjugate of the regularizer, the admissible dual vector fields are given by
\bfln
\mathcal{K} := \{ (\phi_x, \phi_t) \in & C_0(\Omega\times\R; \R^d\times\R ):  \\ & \phi_t(x,t) + \rho(x,t) \geq \eta^*(\phi_x(x,t)), \qquad \forall (x,t) \in \Omega\times\R \} .
\label{eq:constraint_set_K}
\efln
In \cite{lifting_global_solutions} the authors show that $F(u) = \mathcal{F}(\bs{1}_u)$ holds for any $u\in W^{1,1}$. Moreover, 
if the non-convex set $ \{ \bs{1}_u : u\in W^{1,1}  \} $ is relaxed to the convex set
\bfln
C := \{ v \in & \BV_{\text{loc}} (\Omega\times\R, [0,1]): \\ & v(x,t) = 1\; \forall t \leq \min(\Gamma ), \qquad v(x, t) = 0\; \forall t > \max(\Gamma)  \},\label{eq:setc}
\efln 
any minimizer of the lifted problem $\inf_{v\in C}\mathcal{F}(v)$ can be transformed into a global minimizer of the original nonconvex problem $\inf_{u\in W^{1,1}}\mathcal{F}(\bs{1}_u)$  by  thresholding.

In practice, the discretization of the label space $\Gamma$ during the implementation process leads to artifacts and the quality of the solution strongly depends on the number and positioning of the chosen discrete labels. Therefore, it is advisable to employ a \emph{sublabel-accurate} discretization \cite{sublabel_cvpr}, which allows to preserve information about the data term in between discretization points, resulting in smaller problems.
In \cite{sublabel_discretization} the authors point out that this approach is closely linked to the approach in \cite{lifting_global_solutions} when a combination of piecewise linear and piecewise constant basis functions is used for discretization.

More recent developments in the field of functional lifting include an extension to the sublabel-accurate lifting approach to arbitrary convex regularizers \cite{mollenhoff2019lifting} and a connection to Dynamical Optimal Transport and the Benamou-Brenier formulation that also allows to incorporate higher-order regularization \cite{vogt2020connection}. 
   
\mypar{Notation.}
We denote the \emph{extended real line} as $\Rc := \R \cup \{ \pm \infty\}$. 
Given a function $f:\R^n\mapsto\Rc$ the conjugate $f^*:\R^n\mapsto\Rc$ is defined as \cite[Ch.~11]{book_rock_variational} 
\begin{equation}
f^*(u^*) := \sup_{u\in\R^n} \{ \langle u^*, u \rangle -f(u) \}.
\end{equation}
If $f$ has a proper convex hull, both the conjugate and biconjugate are proper, lower semi-continuous and convex. The \emph{indicator function} of a set $C$ is defined as $\delta_C(x):=0$ if $x\in C$ and $+\infty$ otherwise. Whenever $u$ denotes a vector, we use subscripts $u_k$ to indicate an iteration or sequence, and superscripts $u^k$ to indicate the $k$-th value of the vector.

\section{Sublabel-Accurate Lifting Approach}
\label{subsec:lifting}

For reference, we provide a short summary of the lifting approach with sublabel-accurate discretization for TV-\eqref{eq:problem} problems using the notation from \cite{sublabel_cvpr}. The approach comprises three steps:

\mypar{Lifting of the label space.} First, we choose $L$ labels $ \gamma_1 < \gamma_2 < ... < \gamma_L $ such that $ \Gamma = [ \gamma_1, \gamma_L ] $. These labels decomposese the label space $\Gamma$ into $l := L-1$ \emph{sublabel spaces} $\Gamma_i := [\gamma_i, \gamma_{i+1}]$. Any value in $\Gamma$ can be written as
\begin{equation}
\gamma_i^{\alpha} := \gamma_i + \alpha (\gamma_{i+1} - \gamma_i ),
\label{eqdef:sa_gammaia}
\end{equation}
for some $i \in \{1, 2, ..., l\}$ and $\alpha \in [0,1]$. The lifted representation of such a value in $\R^l$ is defined as
\begin{equation}
\bs{1}_i^{\alpha} := \alpha \bs{1}_i + (1-\alpha )\bs{1}_{i-1},
\end{equation}
where $\bs{1}_i \in \R^l$ is the vector of $i$ ones followed by $l-i$ zeroes. The -- non-convex -- \emph{lifted label space} is given as
$\boldsymbol{\Gamma} := \{ \bs{1}_i^{\alpha}  \in\R^l |  i \in \{ 1,2,...,l \}, \alpha \in [0,1] \}$. 
Any lifted value $\bs{u}(x) = \bs{1}_i^{\alpha} \in \boldsymbol{\Gamma}$ can be  mapped uniquely to the equivalent value in the unlifted label space by applying 
\begin{equation}
u(x) = \gamma_1 + \sum_{i=1}^l \bs{u}^i(x)(\gamma_{i+1} - \gamma_i ).
\label{eq:projection}\end{equation}
We refer to such functions $\bs{u}$ as \emph{sublabel-integral.}

\mypar{Lifting of the data term.} Next, a lifted formulation of the data term is derived that 
in effect approximates the energy locally convex between neighboring labels. 
For the possibly non-convex data term of \eqref{eq:problem}, the lifted -- yet still non-convex -- representation for fixed $x\in\Omega$ is defined as $\bs{\rho}:\R^l \mapsto \Rc$,
\begin{align}
\bs{\rho}(\bs{u}) := \inf_{i\in\{1,...,l\}, \alpha \in [0,1]}  \left \{ \rho(\gamma_i^{\alpha}) + \delta_{\bs{1}_i^{\alpha}}(\bs{u}) \right \}.
\label{eq:def_data_integrand}
\end{align}
Note that the domain is $\R^l$ and not just $\bs{\Gamma}$. Outside of the lifted label space $\boldsymbol{\Gamma}$  the lifted representation $\bs{\rho}$ is set to $\infty$. Applying the definition of Legendre-Fenchel conjugates twice to the integrand of the data term results in a relaxed -- and convex -- data term:
\begin{equation}
 \bs{H}(\bs{u}) = \int_{\Omega} \bs{\rho}^{**}(x, \bs{u}(x)) dx.
\label{eq:def_data}
\end{equation}
For explicit expressions of $\bs{\rho}^{**}$ in the linear and non-linear case we refer  to \cite[Prop.~1, Prop.~2]{sublabel_cvpr}.

\mypar{Lifting of the total variation regularizer.} Lastly, a lifted representation of the (isotropic) total variation regularizer is established, building on the theory developed in the context of multiclass labeling approaches \cite{lifting_continuous_multiclass,lifting_tv_local_envelope}.
For fixed $x \in \Omega$ the lifted -- and non-convex -- integrand $\boldsymbol{\phi}: \R^{l\times d} \mapsto \Rc$ is defined:
\begin{align}
\boldsymbol{\phi}(\bs{g}) := &\inf_{1\leq i\leq j\leq l, \alpha, \beta \in [0,1]} 
| \gamma_i^{\alpha} - \gamma_j^{\beta} | \cdot \| v \|_2 + \delta_{(\bs{1}_i^{\alpha} - \bs{1}_j^{\beta}) v^{\top}}(\bs{g}) . 
\label{eq:def_reg_integrand}
\end{align}
Applying the definition of Legendre-Fenchel conjugates twice to the lifted integrand of the regularizer results in a relaxed -- and convex -- regularization term:
\begin{flalign}
\bs{TV}(\bs{u}) := \int_{\Omega} \boldsymbol{\phi}^{**}(D\bs{u}),
\label{eq:def_reg}
\end{flalign}
where $D\bs{u}$ is the distributional derivative in the form of a Radon measure. For isotropic TV, it can be shown that for  $\bs{g}\in\R^{l\times d}$, %as 
\begin{flalign}
\bs{\phi}^{**}(\bs{g}) &= \sup_{\bs{q}\in\mathcal{K}_{\text{iso}}} \langle \bs{q}, \bs{g} \rangle , \\
\mathcal{K}_{\text{iso}} &= \left\{ \bs{q}\in\R^{l\times d} \quad  \middle| \quad \|  \bs{q}_i \|_2 \leq \gamma_{i+1} -\gamma_i, \quad \forall i=1,...,l \right\}.
\label{eq:K_iso}
\end{flalign}  
For more details we refer to \cite[Prop.~4]{sublabel_cvpr} and \cite{lifting_tv_local_envelope}. 
Unfortunately isotropic TV in general does not allow to prove global optimality for the discretized system. Therefore we also consider the lifted anisotropic ($L^1$) TV, by replacing \eqref{eq:K_iso} with
\begin{flalign}
\label{eq:K_an}
\mathcal{K}_{\text{an}} &= \left\{ \bs{q}\in\R^{l\times d} \quad \middle| \quad \| \bs{q}_i \|_\infty \leq \gamma_{i+1} -\gamma_i, \quad \forall i = 1,...,l \right\}\\
&= \bigcap_{j=1, \hdots  ,d} \left\{  \bs{q} \in \R^{l\times d} \quad \middle| \quad  \| \bs{q}_{i,j} \|_2 \leq \gamma_{i+1} - \gamma_i , \quad \forall i=1,...,l \right\} .
\end{flalign} 

Together, the previous three sections allow us to formulate a version of the problem of minimizing the lifted energy \eqref{eq:lifted_F} over the relaxed set \eqref{eq:setc} that is discretized in the label space $\Gamma$:
\begin{equation}
\inf_{\bs{u}\in\BV(\Omega,\bs{\Gamma})} \int_{\Omega} \bs{\rho}^{**}(x, \bs{u}(x)) +  \int_{\Omega} \boldsymbol{\phi}^{**}(D\bs{u}).
\end{equation}
Once the non-convex set $\bs{\Gamma}$ is relaxed to its convex hull, we obtain a fully convex lifting of problem TV-\eqref{eq:problem} similar to \eqref{eq:prob-lf}, which can now be spatially discretized.

\section{Equivalency of the Lifted Bregman Iteration}

This chapter addresses the question under which conditions Alg.~1 and Alg.~2 are equivalent. We stipulate a sufficient condition on the subgradients used in the Bregman iteration and prove in chapter 4 that this condition is met in case of the anisotropic TV regularizer. The key idea is to note that the Bregman iteration amounts to extending the data term by a linear term, and that the sum of the separately relaxed terms is point-wise equal to the relaxation of their sum. Note that this additivity does not hold for general sums.

The following considerations are formal due to the mostly pointwise arguments; we leave a rigorous investigation in the function space to future work. However, they can equally be understood in the spatially discrete setting with finite $\Omega$, where arguments are more straightforward. For readability, we consider a fixed $x \in \Omega $ and omit $x$ in the arguments.

\begin{proposition}
Assume $\rho_1,\rho_2,h: \Gamma \mapsto \Rc$ with
\begin{align}
\rho_2(u) := \rho_1 (u) - h(u), \qquad h(u) := pu, \qquad p\in\R,
\end{align}
where $\rho_1$ and $\rho_2$ should be understood as two different  data terms in~\eqref{eq:problem}. Define
\begin{equation}
\tilde{\bs{\gamma}} := \begin{pmatrix} \gamma_2 - \gamma_1, & \hdots, & \gamma_{L} - \gamma_l \end{pmatrix}^{\top}\label{eq:bs_gamma}
\end{equation}
Then, for the lifted representations $\bs{\rho}_1,\bs{\rho}_2,\bs{h} : \R^l \mapsto \Rc$ in \eqref{eq:def_data_integrand}, it holds 
\begin{equation}
\bs{\rho}_2^{**}(\bs{u}) \quad = \quad \bs{\rho}_1^{**}(\bs{u}) - \bs{h}^{**}(\bs{u}) \quad  = \quad  \bs{\rho}_1^{**}(\bs{u}) - \langle p \tilde{\bs{\gamma}}, \bs{u} \rangle  .
\end{equation}
\label{prop:lifted_sum}
\end{proposition}
\begin{proof}[Proof of Proposition \ref{prop:lifted_sum}]
The proof is slightly technical and we only sketch it. By definition of the Fenchel conjugate and after some transformations, $\bs{\rho}_2^*$ becomes
\begin{align}
\bs{\rho}_2^*(\bs{v}) = \sup_{j\in\{1,...,l\}, \beta \in [0,1]} \left \{ \langle \bs{1}_j^{\beta}, \bs{v} + p \tilde{\bs{\gamma}} \rangle - \rho_1(\gamma_j^{\beta}) \right \}.
\end{align}
Applying the definition of the Fenchel conjugate once again eventually leads to 
\begin{align}
\bs{\rho}_2^{**}(\bs{u}) = \bs{\rho}_1^{**}(\bs{u}) - \langle  p \tilde{\bs{\gamma}} , \bs{u} \rangle .
\end{align}
Comparing this to \cite[Prop.~2]{sublabel_cvpr} we see that $\langle  p \tilde{\bs{\gamma}} , \bs{u} \rangle = \bs{h}^{**}(\bs{u})$.
\end{proof}

The following proposition shows that Alg.~1 and Alg.~2 are equivalent as long as we base the iteration on subgradients $p_{k-1}$ and $\bs{p}_{k-1}$ in the subdifferential of $J(u_{k-1})$ and $\bs{J}(\bs{u}_{k-1})$ that are linked in a particular way.

\begin{proposition} Assume that the minimization problems \eqref{eq:breg_uk} in the original Bregman iteration have unique solutions. Moreover, assume that in the lifted iteration, the solutions $\bs{u}_k$ of \eqref{eq:breg_uk_lifted} in each step satisfy $\bs{u}(x)\in\bs{\Gamma}$, i.e., are sublabel-integral. If at every point $x$ the chosen subgradients $p_{k-1}\in\partial J(u_{k-1})$ and $\bs{p}_{k-1} \in\partial \bs{J}(\bs{u}_{k-1})$ satisfy 
\bfl 
\bs{p}_{k-1}(x) \quad = 
\quad p_{k-1}(x) \tilde{\bs{\gamma}} 
\label{eq:lifted_grad}
\efl
with $\tilde{\bs{\gamma}}$ as in \eqref{eq:bs_gamma}, then the lifted iterates $\bs{u}_k$ correspond to the iterates $u_k$ of the classical Bregman iteration \eqref{eq:breg_uk} according to \eqref{eq:projection}.
\label{prop:equivalency}
\end{proposition}

\begin{proof}[Proof of Proposition \ref{prop:equivalency}]
We define the \emph{extended data term}
\begin{equation}
\tilde{H}(u) := \into \rho(x,u(x)) - p(x) u(x) \dx,
\end{equation}
which incorporates the linear term of the Bregman iteration.
Using Prop.~\ref{prop:lifted_sum}, we reach the following lifted representation:
\begin{align}
\tilde{\bs{H}}(\bs{u}) &= \into \bs{\rho}^{**}(x,\bs{u}(x)) - \langle p(x)\tilde{\bs{\gamma}}, \bs{u}(x) \rangle \dx .
\end{align}
Hence the lifted version of \eqref{eq:breg_uk} is
\begin{equation}
\arg\min_{\bs{u}\in\bs{U}} \left \{ \bs{H}(\bs{u}) +\bs{J}(\bs{u}) - \langle p_{k-1} \tilde{\bs{\gamma}}, \bs{u} \rangle  \right \}.
\end{equation}
Comparing this to  \eqref{eq:breg_uk_lifted} shows that the minimization problem in the lifted iteration is the lifted version of \eqref{eq:breg_uk} if the subgradients  $p_{k-1}\in\partial J(u_{k-1})$ and $\bs{p}_{k-1} \in\partial \bs{J}(\bs{u}_{k-1})$ satisfy $ \bs{p}_{k-1} = p_{k-1} \tilde{\bs{\gamma}}$.
 In this case, since we have assumed that the solution of the lifted problem \eqref{eq:breg_uk_lifted} is sublabel-integral, it can be associated via \eqref{eq:projection} with the solution of the original problem \eqref{eq:breg_uk}, which is unique by assumption. 
\end{proof}
Thus, under the condition of the proposition, the lifted and unlifted Bregman iterations are equivalent.
\section{Numerical Discussion and Results}
In this section, we consider the spatially discretized problem on a finite discretized domain~$\Omega^h$ with grid spacing $h$. 
In particular, we will see that the subgradient condition in Prop.~\ref{prop:equivalency} can be met in case of anisotropic TV and how such subgradients can be obtained in practice.

\mypar{Finding a subgradient.}
The discretized, sublabel-accurate relaxed total variation is  of the form
\begin{flalign}
        J^h(\nabla \bs{u}^h) &= \max_{\bs{q}^h:\Omega^h \rightarrow\R^{k\times d}} \left \{ \sum_{x \in \Omega^h}  \langle \bs{q}^h(x), \nabla \bs{u}^h(x) \rangle - \delta_\mathcal{K}(\bs{q}^h(x)) \right \},
\label{eq:q}
\end{flalign}
with $\mathcal{K}$ defined by \eqref{eq:K_iso} or \eqref{eq:K_an} and $\nabla$ denoting the discretized forward-difference operator. By standard convex analysis (\cite[Thm.~23.9]{book_rock_convex}, \cite[Cor.~10.9]{book_rock_variational}, \cite[Prop.~11.3]{book_rock_variational}) we can show that if $\bs{q}^h$ is a  maximizer of \eqref{eq:q}, then $\bs{p}^h := \nabla^{\top}\bs{q}^h$ is a subgradient of $J^h(\nabla \bs{u}^h)$. Thus, the step of choosing a subgradient \eqref{eq:breg_qk_lifted} boils down to $\bs{p}_k^h = \nabla^\top \bs{q}_k^h$ and for the dual maximizer $\bs{q}_{k-1}^h$ of the last iteration we implement  \eqref{eq:breg_uk_lifted} as:
\begin{flalign}
\bs{u}^h_k = \arg\min_{\bs{u}^h:\Omega^h \mapsto \R^l} \max_{\bs{q}^h_k:\Omega^h \mapsto\mathcal{K}} \sum_{x\in\Omega^h} (\bs{\rho}^h)^{**}(x,\bs{u}^h(x)) + \langle \bs{q}^h_k - \bs{q}^h_{k-1}, \nabla \bs{u}^h \rangle .
\end{flalign}

\mypar{Transforming the subgradient.}
In Prop.~\ref{prop:equivalency} we formulated a constraint on the subgradients for which the original and lifted Bregman iteration are equivalent. While this property is not necessarily satisfied if the subgradient $\bs{p}^h_{k-1}$ is chosen according to the previous paragraph, we will now show that any such subgradient can be transformed into another valid subgradient that satisfies condition \eqref{eq:lifted_grad}.

Consider a pointwise sublabel-integral solution $\bs{u}^h_k$ with subgradient \linebreak $\bs{p}^h_k := \nabla^\top \bs{q}^h_k \in \partial \bs{J}^h(\bs{u}^h_k)$ for $\bs{q}^h_k(\cdot) \in \mathcal{K}$ being a maximizer of \eqref{eq:q}. We define a pointwise transformation: For fixed $x^m \in \Omega^h$ and $\bs{u}^h_k(x^m) = \bs{1_i^\alpha}$, let $(\bs{q}^h_k(x^m))^i \in \R^d$ denote the $i$-th row of $\bs{q}^h_k(x^m)$ corresponding to the $i$-th label as prescribed by $\bs{u}^h_k(x^m) = \bs{1}_i^\alpha$. Both in the isotropic and anisotropic case the transformation
\bfl 
        \tilde{\bs{q}}^h_k(x^m) := \frac{(\bs{q}^h_k(x^m))^i}{\gamma_{i+1} - \gamma_i} \tilde{\bs{\gamma}}
        \label{eq:gradient_transform}
\efl
returns an element of the set $\mathcal{K}$, i.e., $\mathcal{K}_{\text{iso}}$ or $\mathcal{K}_{\text{an}}$. In the anisotropic case we can furthermore show that $\tilde{\bs{q}}^h_k$ also maximizes \eqref{eq:q} and therefore the transformation gives a subgradient $\tilde{\bs{p}}^h_k := \nabla^\top \tilde{ \bs{q}}^h_k \in \partial \bs{J}^h(\bs{u}^h_k)$ of the desired form \eqref{eq:lifted_grad}:

\begin{proposition} 
Consider the anisotropic TV-regularized case \eqref{eq:K_an}. Assume that the  iterate $\bs{u}^h_{k}$ is sublabel-integral. Moreover, assume that $\bs{p}^h_k := \nabla^\top \bs{q}^h_k$ is a subgradient in $\partial \bs{J}^h(\bs{u}^h_{k})$ and define $\tilde{\bs{q}}^h_k$ pointwise as in \eqref{eq:gradient_transform}. Then $\tilde{\bs{p}}^h_k := \nabla^\top \tilde{ \bs{q}}^h_k$ is also a subgradient and furthermore of the form
\bfl 
\tilde{\bs{p}}^h_k \quad = \quad 
p^h_{k} \tilde{\bs{\gamma}}^h,\label{eq:ptildapkm1}
\efl
where $p^h_{k}$ is a subgradient in the unlifted case, i.e., $p^h_{k}\in\partial J^h(u^h_{k})$.
\label{prop:gradient_an}
\end{proposition}

\begin{proof}[Proof of Proposition \ref{prop:gradient_an}]

In the anisotropic case the spatial dimensions are uncoupled, therefore w.l.o.g. assume $d=1$. Consider two neighboring points $x^m$ and $x^{m+1}$ with $\bs{u}^h_k(x^m) = \bs{1}_i^\alpha$ and $\bs{u}^h_k(x^{m+1}) = \bs{1}_j^\beta$. Applying the forward difference operator gives 
\begin{flalign}
\nabla \bs{u}^h_k(x^m) = \frac{1}{h}\begin{cases}
        (\bs{0}_{i-1}, \quad 1-\alpha, \quad \bs{1}_{j-i-2}, \quad \beta, \quad \bs{0}_{l-j})^\top, & i < j, \\
        (\bs{0}_{i-1}, \quad \beta - \alpha, \quad \bs{0}_{l-i})^\top, & i = j, \\
        (\bs{0}_{j-1}, \quad \beta -1, \quad -\bs{1}_{i-j-2}, \quad -\alpha, \quad \bs{0}_{l-j})^\top, & i > j. \\
\end{cases}
\end{flalign}
Maximizers $\bs{q}^h_k(x^m) \in \mathcal{K}_{an}$ of the dual problem \eqref{eq:q} are exactly all vectors 
\begin{flalign}
\bs{q}^h_k(x^m) = \begin{cases}
        (***, \quad \gamma_{i+1}-\gamma_i, \quad ... , \quad \gamma_{j+1} - \gamma_j, \quad ***)^\top, & i < j, \\
        (***, \quad \sgn(\beta - \alpha) (\gamma_{i+1}-\gamma_i), \quad ***)^\top, & i = j, \\
        (***, \quad \gamma_{j} - \gamma_{j+1} , \quad ... , \quad \gamma_i - \gamma_{i+1}, \quad ***)^\top, & i > j. \\
\end{cases}
\label{eq:Kan_minimizers}
\end{flalign}
The elements marked with $*$ can be chosen arbitrarily as long as \linebreak $\bs{q}^h_k(x^m)\in\mathcal{K}_{an}$. Due to this special form, the transformation \eqref{eq:gradient_transform} leads to $ \tilde{\bs{q}}^h_k(x^m) = \pm \tilde{\bs{\gamma}}$  depending on the case. Crucially, this transformed vector is another equally valid choice in~\eqref{eq:Kan_minimizers} and therefore \eqref{eq:gradient_transform} returns another valid subgradient $\tilde{\bs{p}}^h_k = \nabla^\top \tilde{\bs{q}}^h_k$.

In order to show that $p_k^h= \nabla^\top q_k^h$ for $q_k^h(\cdot)= \pm 1$ is a subgradient in the unlifted setting we use the same arguments. To this end, we use  the sublabel-accurate notation with $L=2$. The ``lifted'' label space is  $\bs{\Gamma} = [0,1]$, independently of the actual $\Gamma \subset \R$; see \cite[Prop.~3]{sublabel_cvpr}. Then with $u^h_k(x^m) = \gamma_i^\alpha$ and $u^h_k(x^{m+1}) = \gamma_j^\beta$ (corresponding to $\bs{1}_i^\alpha$ and $\bs{1}_j^\beta$ from before), applying the forward difference operator $\nabla u^h_k(x^m) = \frac{1}{h} ( \gamma_j^\beta - \gamma_i^\alpha )$ shows that dual maximizers are $q^h_k(x^m) = \sgn(\gamma_j^\beta - \gamma_i^\alpha) | \bs{\Gamma} | = \pm 1$. It can be seen that the algebraic signs coincide pointwise in the lifted and unlifted setting.  Thus $p^h_{k}$ in \eqref{eq:ptildapkm1} is of the form $p^h_k = \nabla^\top q_k^h$ and in particular a subgradient in the unlifted setting.

\end{proof}

%%%%%%%%%%%%%%%%%%%  Entropie und Abs Diff  %%%%%%%%%%%%%%%%%%%%%%%%%%%%%%%%%%%%%%

\newcommand{\circimg}[1]{\includegraphics[width=.14\linewidth]{#1}}
\begin{figure}[t]
  \begin{center}
  \begin{tabular}{lllll}
    \circimg{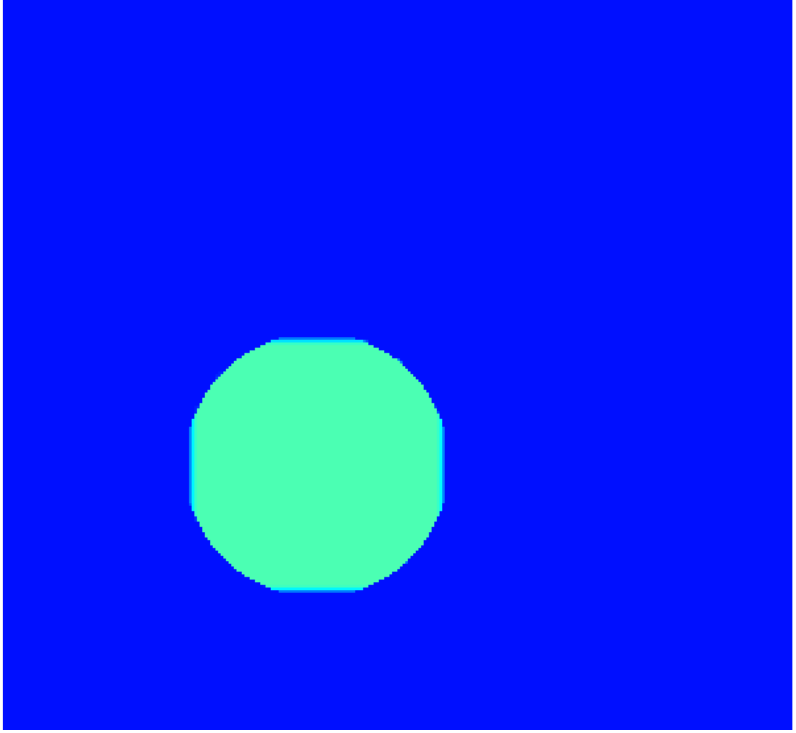}&
    \circimg{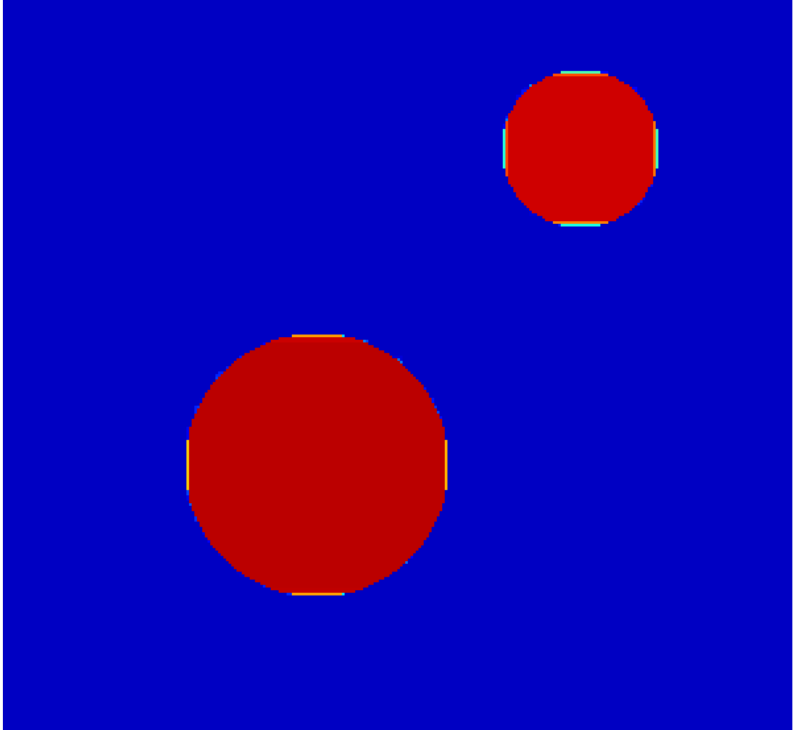}&
    \circimg{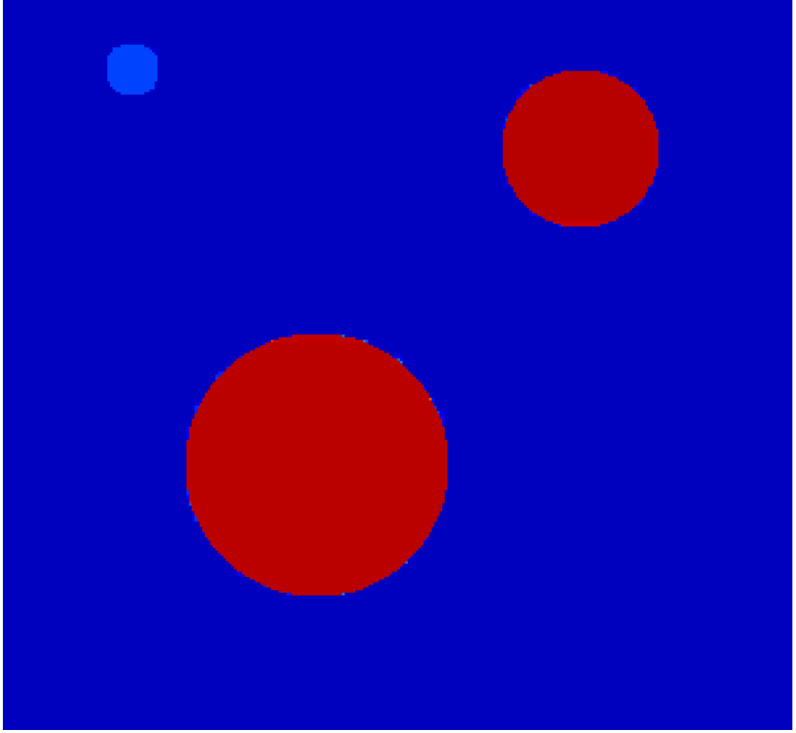}&
    \circimg{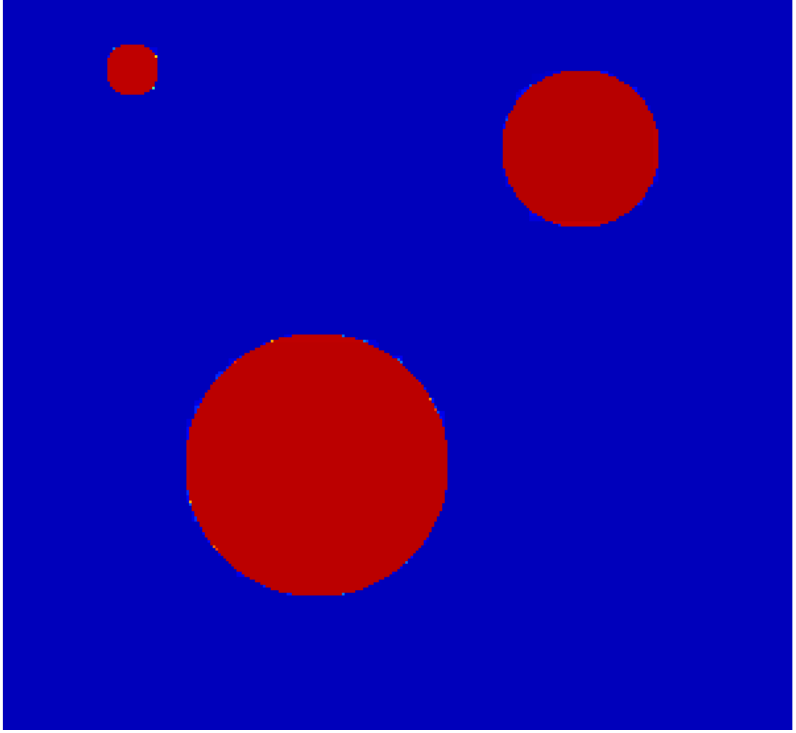}&
    \circimg{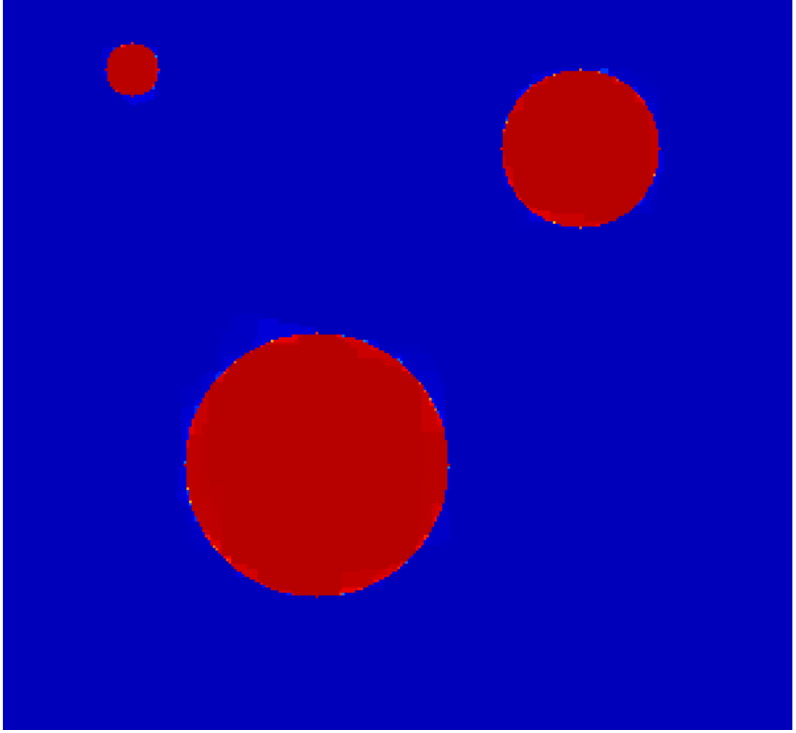}\\
    \circimg{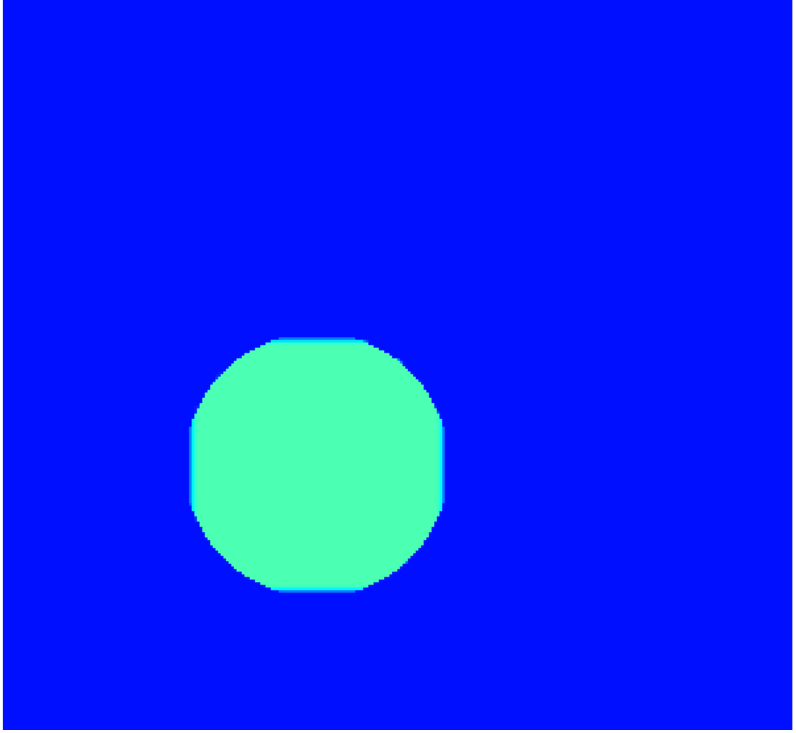}&
    \circimg{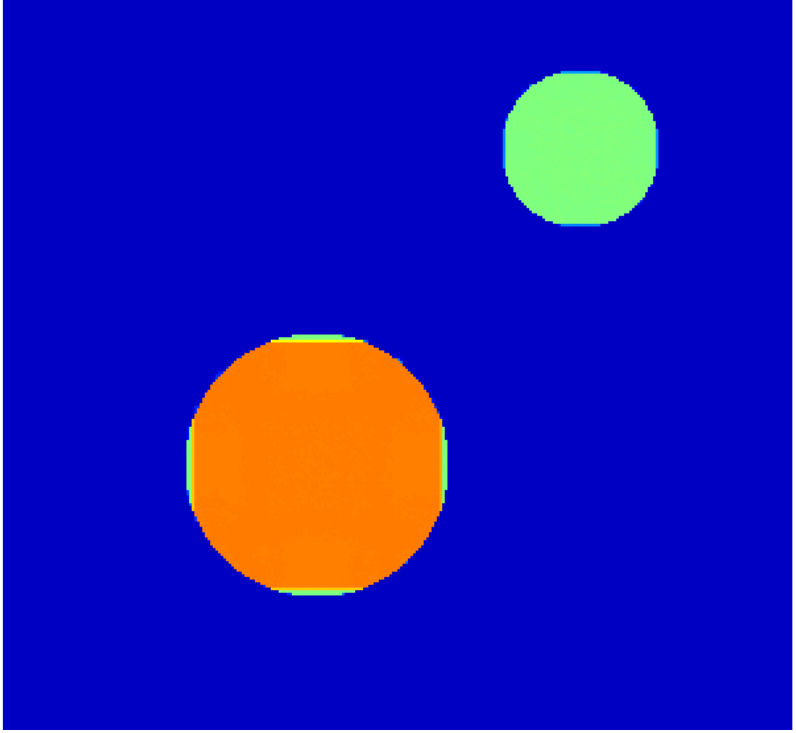}&
    \circimg{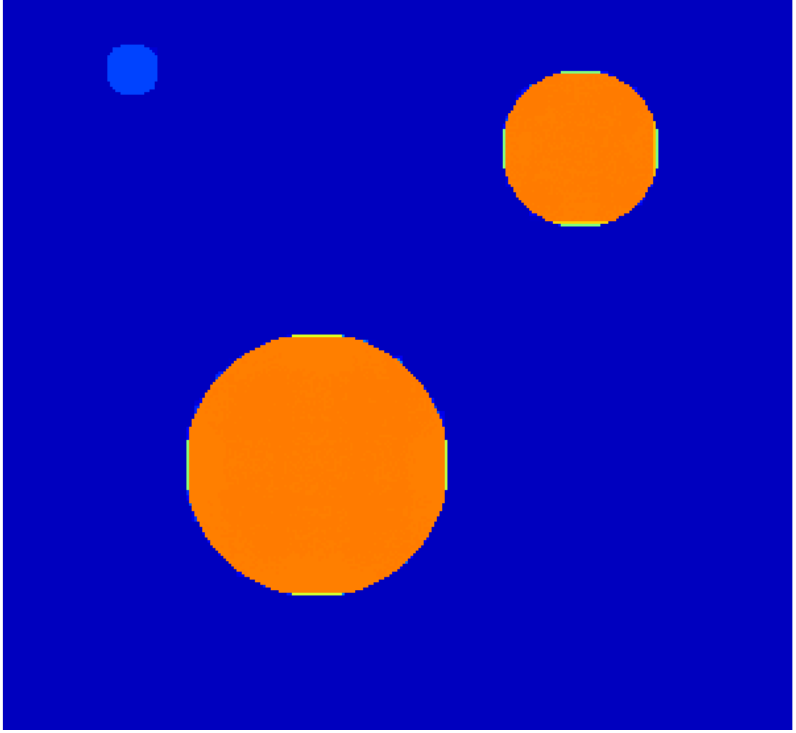}&
    \circimg{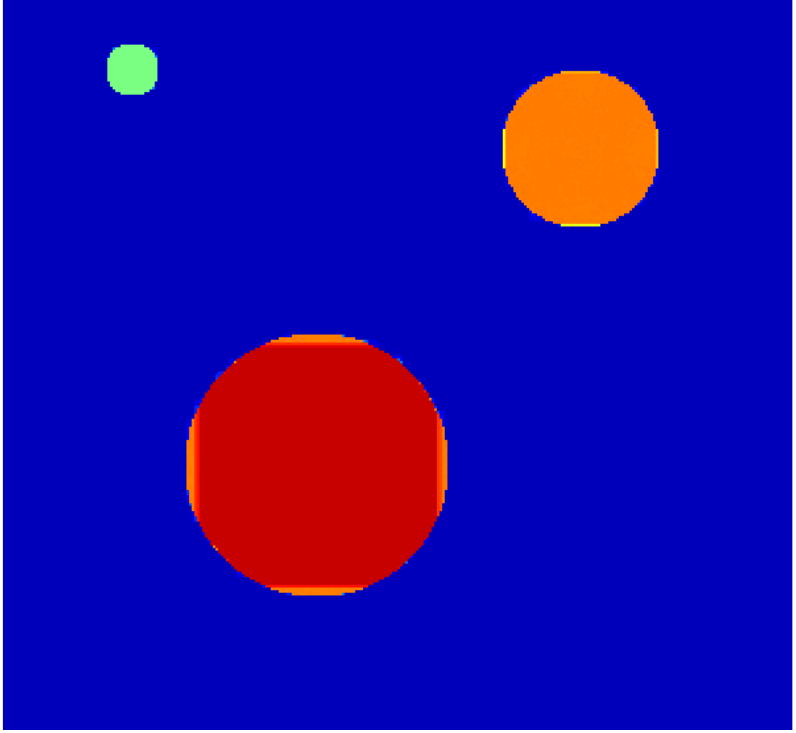}&
    \circimg{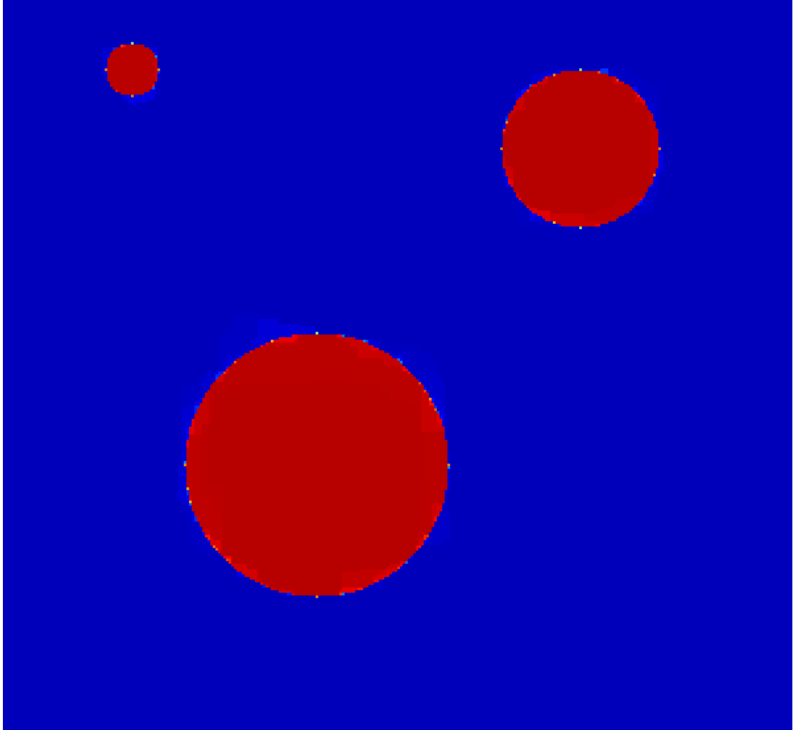}\\
    \circimg{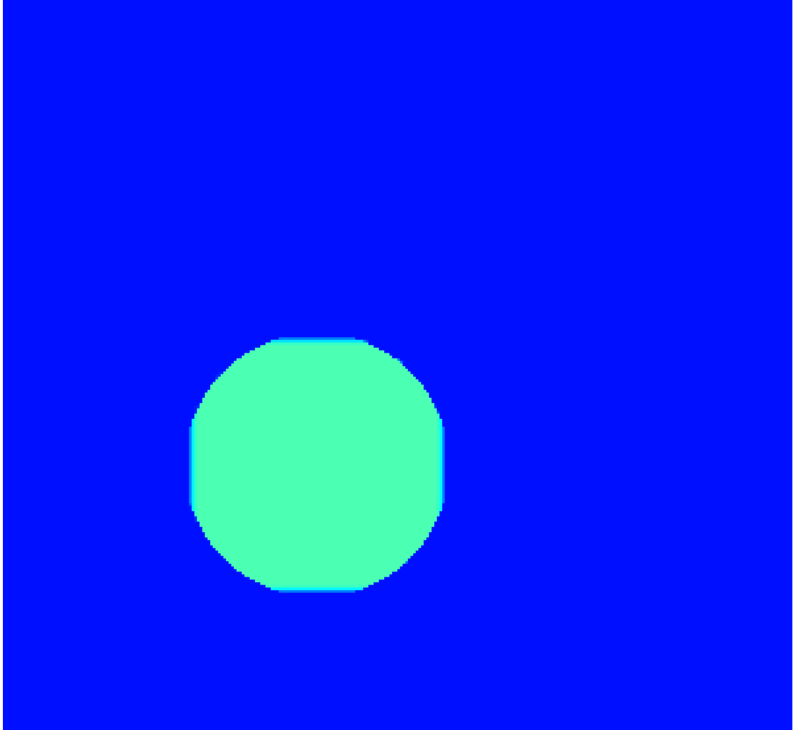}&
    \circimg{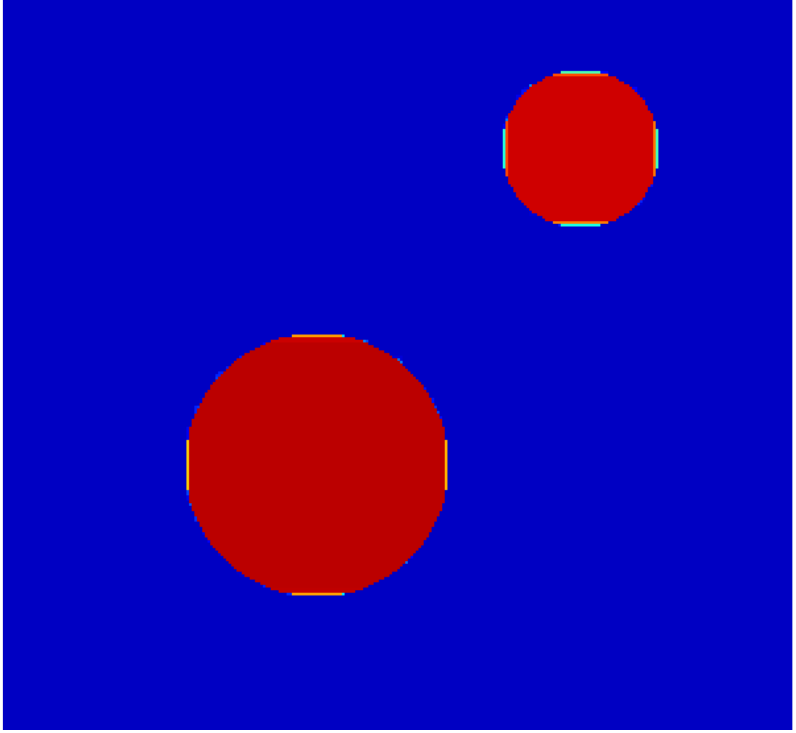}&
    \circimg{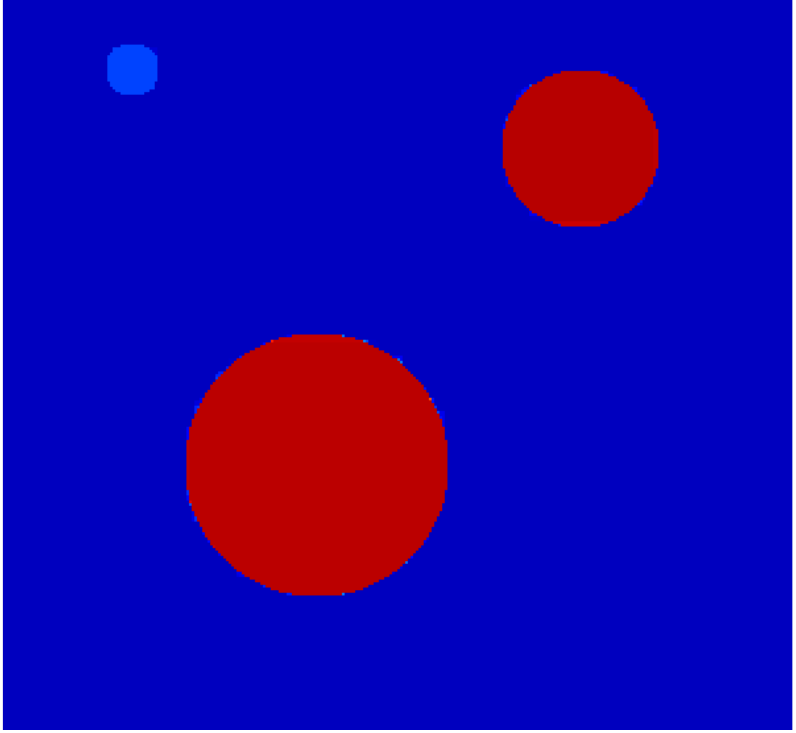}&
    \circimg{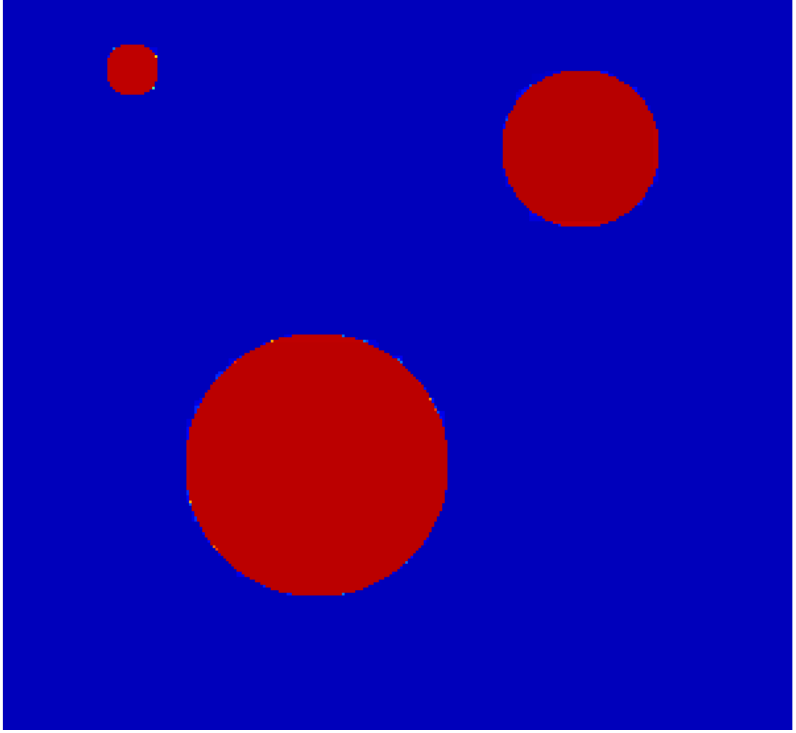}&
    \includegraphics[width=.21\linewidth]{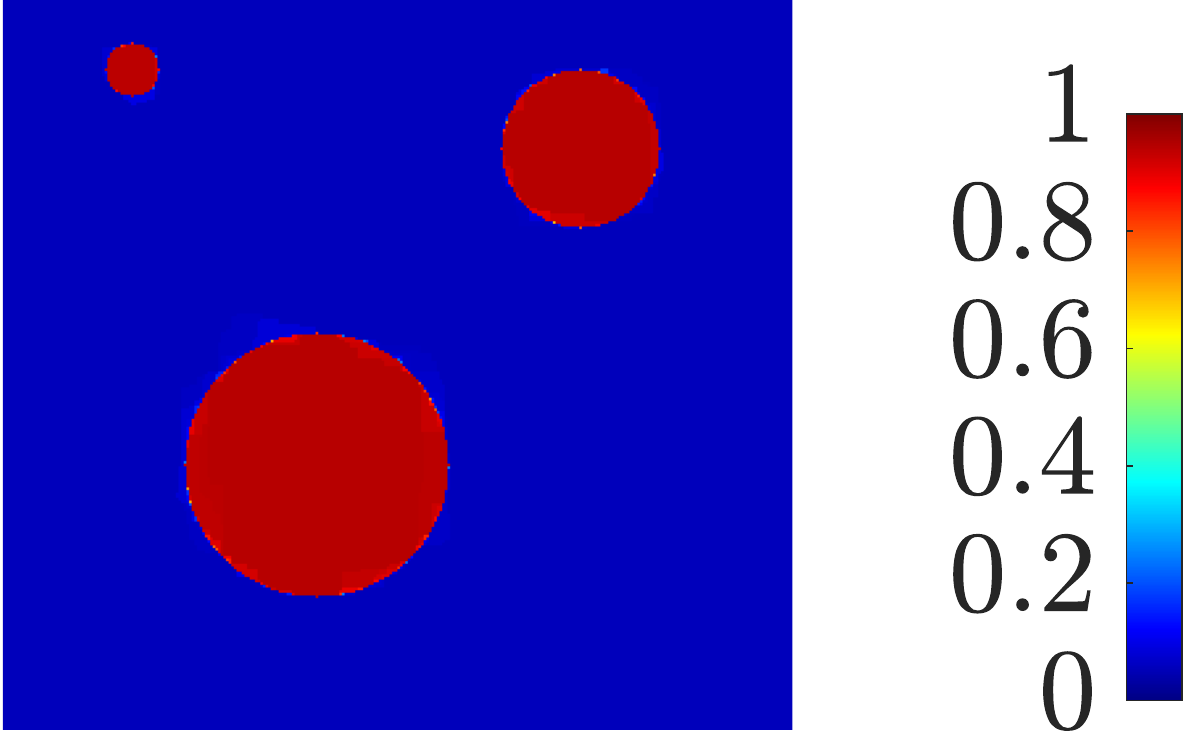}\\
    $\qquad k=1$&
    $\qquad k=2$&
    $\qquad k=3$&
    $\qquad k=4$&
    $\qquad k=50$
  \end{tabular}%
  \end{center}%
\vspace{-1em}\caption{\textbf{Equivalency of classical and lifted Bregman on a convex problem.} On the \emph{convex} ROF-\eqref{eq:problem} problem with anisotropic TV, with a na{\"i}ve implementation, the classical Bregman iteration as in Alg.~1 \textbf{(top row)} and the lifted generalization as in Alg.~2 \textbf{(middle row)} show clear differences. If the lifted subgradients are transformed as in Prop.~\ref{prop:equivalency}, the lifted iterates \textbf{(bottom row)} are visually indistinguishable from the classical iteration. However, the lifted version also allows to transparently handle nonconvex energies (Fig.~\ref{fig:results2}).
%verison  two approaches Comparison of Alg.~1 and Alg.~2. Top row shows the result of the original Bregman iteration, the two bottom rows the results of the lifted Bregman iteration. In the third row we use transformed subgradients. In the latter case  and Alg.~2 are equivalent and return the same results up to negligible numerical differences; i.e. the 2-norm of the difference is of the order 1.0e-02 -- 1.0e-03
}
\label{fig:rof}
\end{figure}

%%%%%%%%%%%%%%%%%%%  Entropie und Abs Diff  %%%%%%%%%%%%%%%%%%%%%%%%%%%%%%%%%%%%%%

\mypar{Convex energy with artificial data.}
We compare the results of the original and lifted Bregman iteration for the ROF-\eqref{eq:problem} problem with $\lambda =20$, synthetic input data and anisotropic TV regularizer. In the lifted setting, we compare implementations with and without transforming the subgradients as in~\eqref{eq:gradient_transform}. 
The results shown in Fig.~\ref{fig:rof} clearly support the theory: Once subgradients are transformed as in Prop.~\ref{prop:equivalency}, the iterates agree with the classical, unlifted iteration.

A subtle issue concerns points where the minimizer of the lifted energy is non-sublabel-integral, i.e., cannot be easily identified with a solution of the original problem. This impedes the recovery of a suitable subgradient as in \eqref{eq:projection}, which leads to diverging Bregman iterations. We found this issue to occur in particular with isotropic TV discretization, which does not satisfy a discrete version of the coarea formula -- which is used to prove in the continuous setting that solutions of the original problem can be recovered by thresholding -- but is also visible to a smaller extent around the boundaries of the objects in Fig.~\ref{fig:rof}. 

\mypar{Non-convex stereo matching with real-world data.} Let us demonstrate the applicability of the lifted Bregman iteration on a non-convex stereo-matching problem for depth estimation.  We use TV-\eqref{eq:problem} with data term 
\bfl \rho(x,u(x)) =   \int_{W(x)} \sum_{d=1,2} h(\partial_{x_d} I_1((y_1,y_2+u(x))) - \partial_{x_d} I_2((y_1,y_2))),
\label{eq:stereo_matching_data_term}
\efl
where $W(x)$ denotes a patch around $x$ and $h(\alpha) := \min \{ \alpha, \beta \} $ is a truncation with threshold $\beta >0$. This data term is non-convex and non-linear in~$u$. We apply the lifted Bregman iteration on three data sets \cite{data_middlebury_bike} with $L=5$ labels, iso-tropic TV regularizer and untransformed subgradients. For results see Fig.~\ref{fig:stereo_profile}, \ref{fig:results2} (Motorbike: $\lambda =20$, $k=30$; Umbrella: $\lambda = 10$; Backpack: $\lambda = 25$). We also ran the experiment with an anisotropic TV regularizer as well as transformed subgradients. Overall, the behavior was similar, but transforming the subgradients led to more pronounced jumps.
Interestingly, even in this non-convex case the solution of the lifted Bregman iteration also strongly reminds of an ISS flow: The first solution is a smooth estimation; as the iteration continues, finer structures are added. This behavior is also visible in the progression of the profiles in Fig.~\ref{fig:stereo_profile}.

\newcommand{\circimgst}[1]{\includegraphics[width=.165\textwidth]{#1}}

\begin{figure}[t]
  \begin{center}
  \begin{tabular}{lllll}
    \circimgst{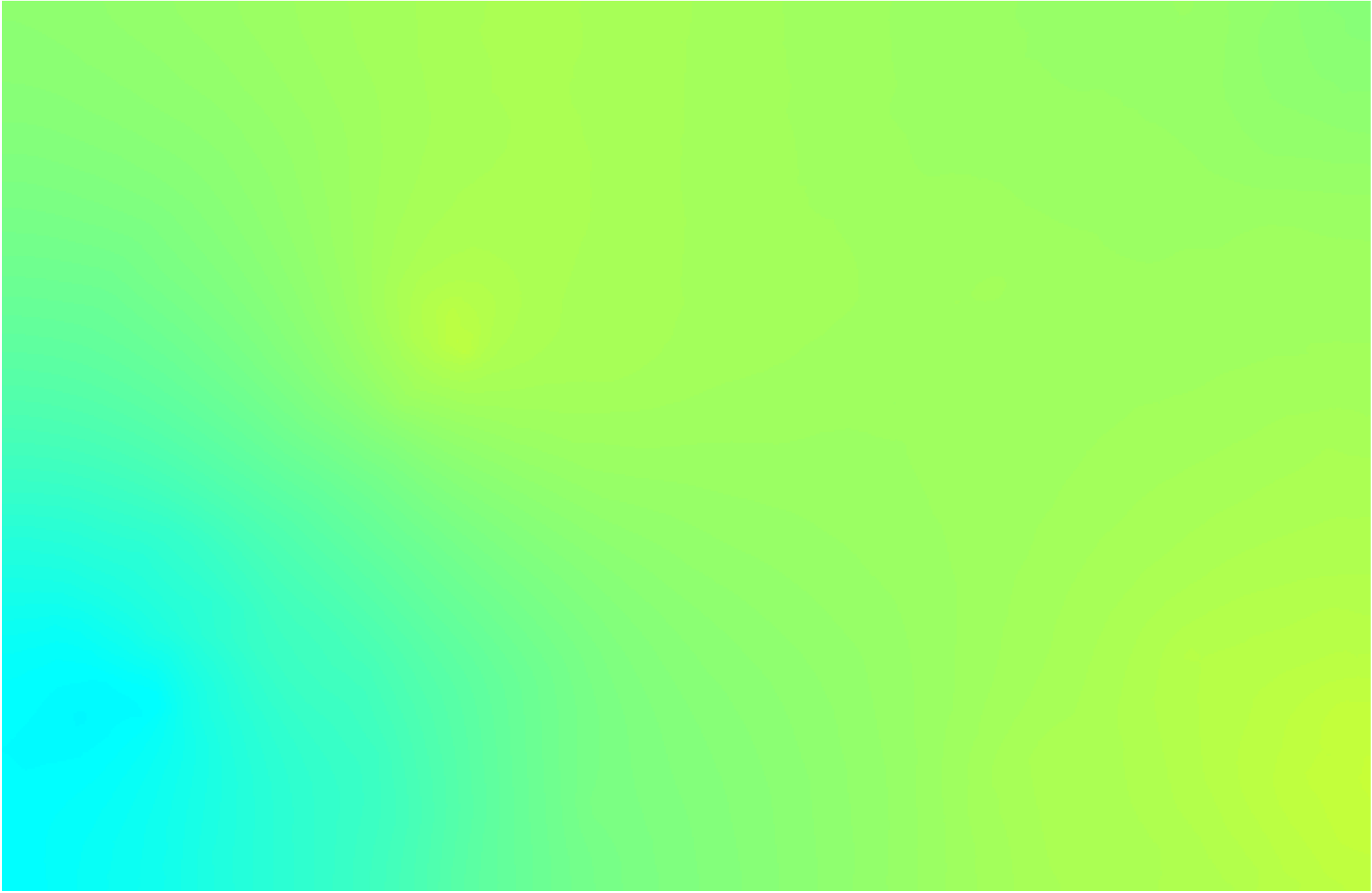}&
    \circimgst{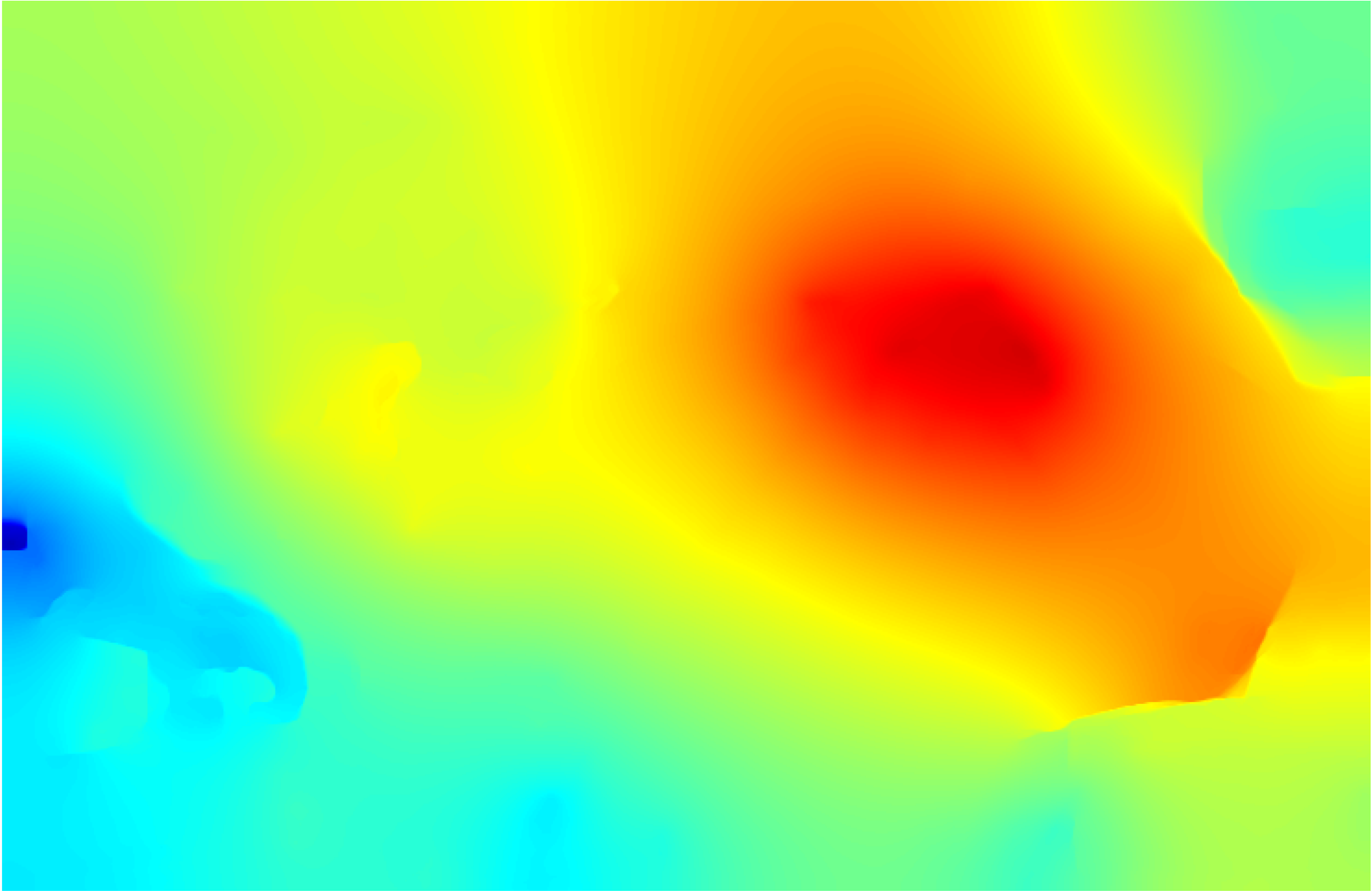}&
    \circimgst{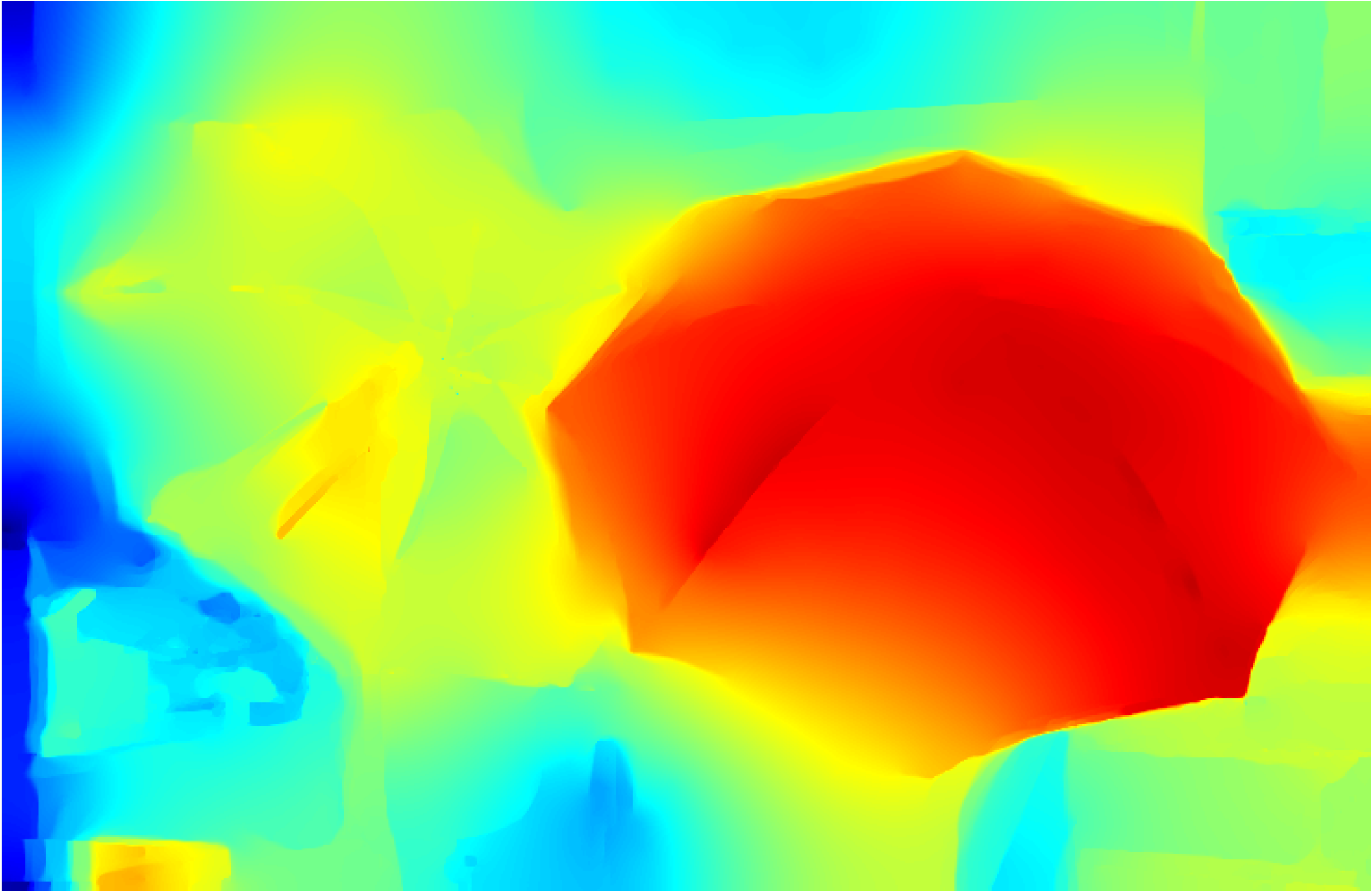}&
    \circimgst{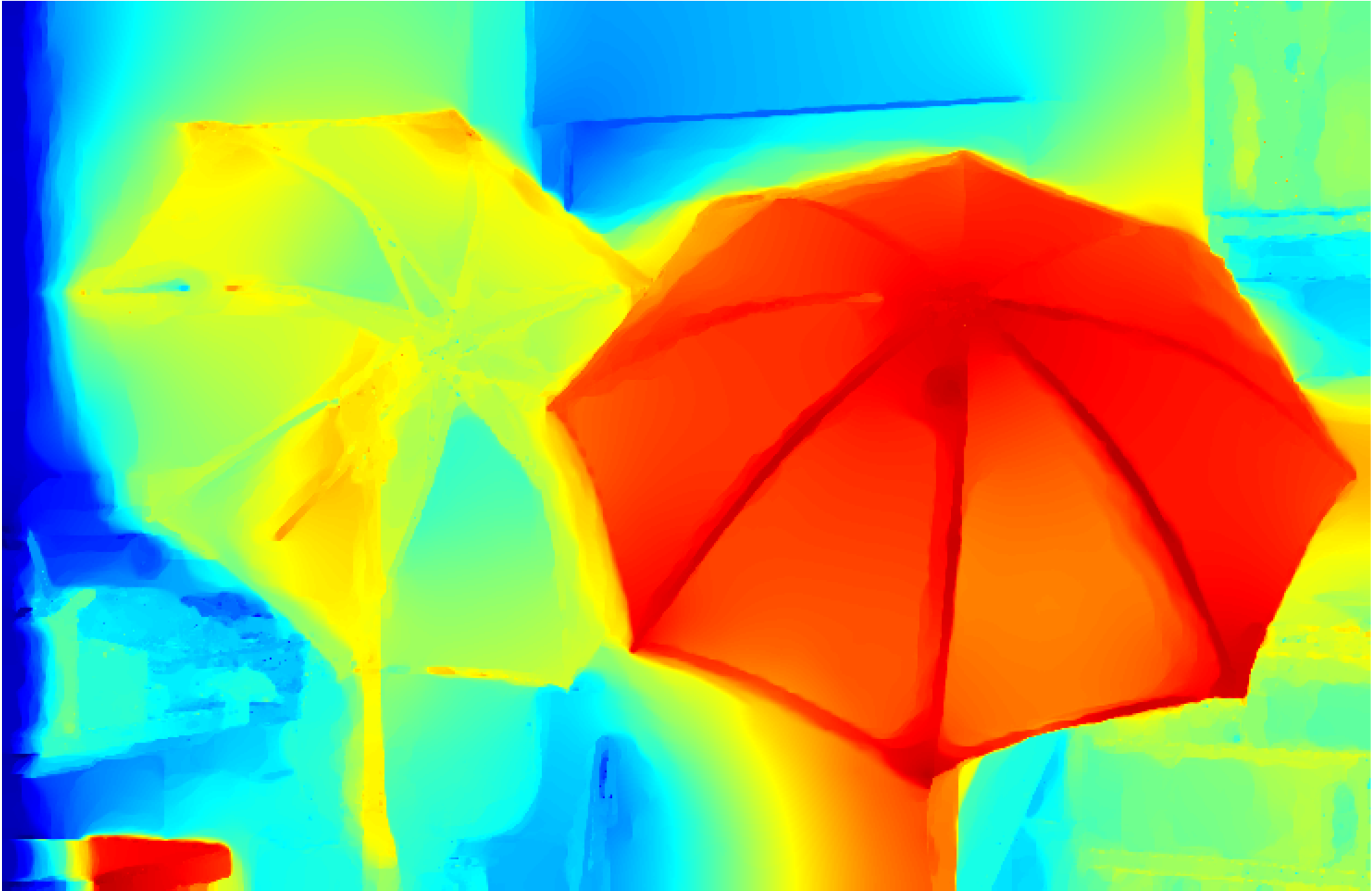} &
    \circimgst{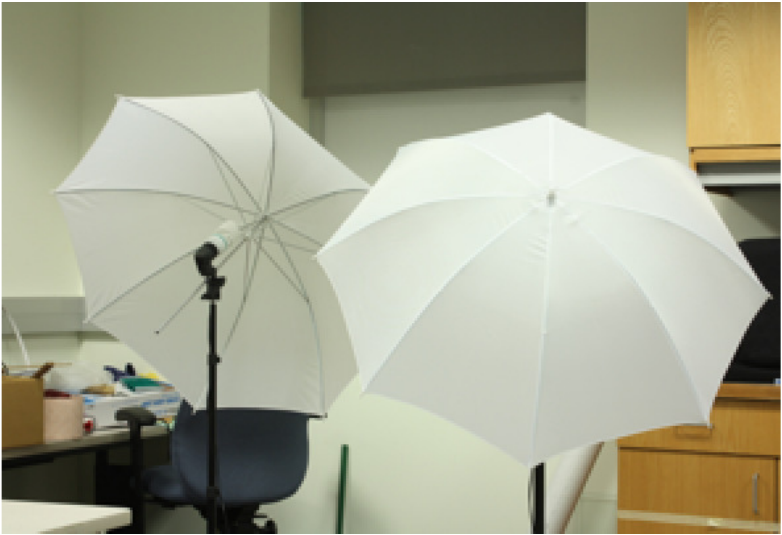} \\   
    \circimgst{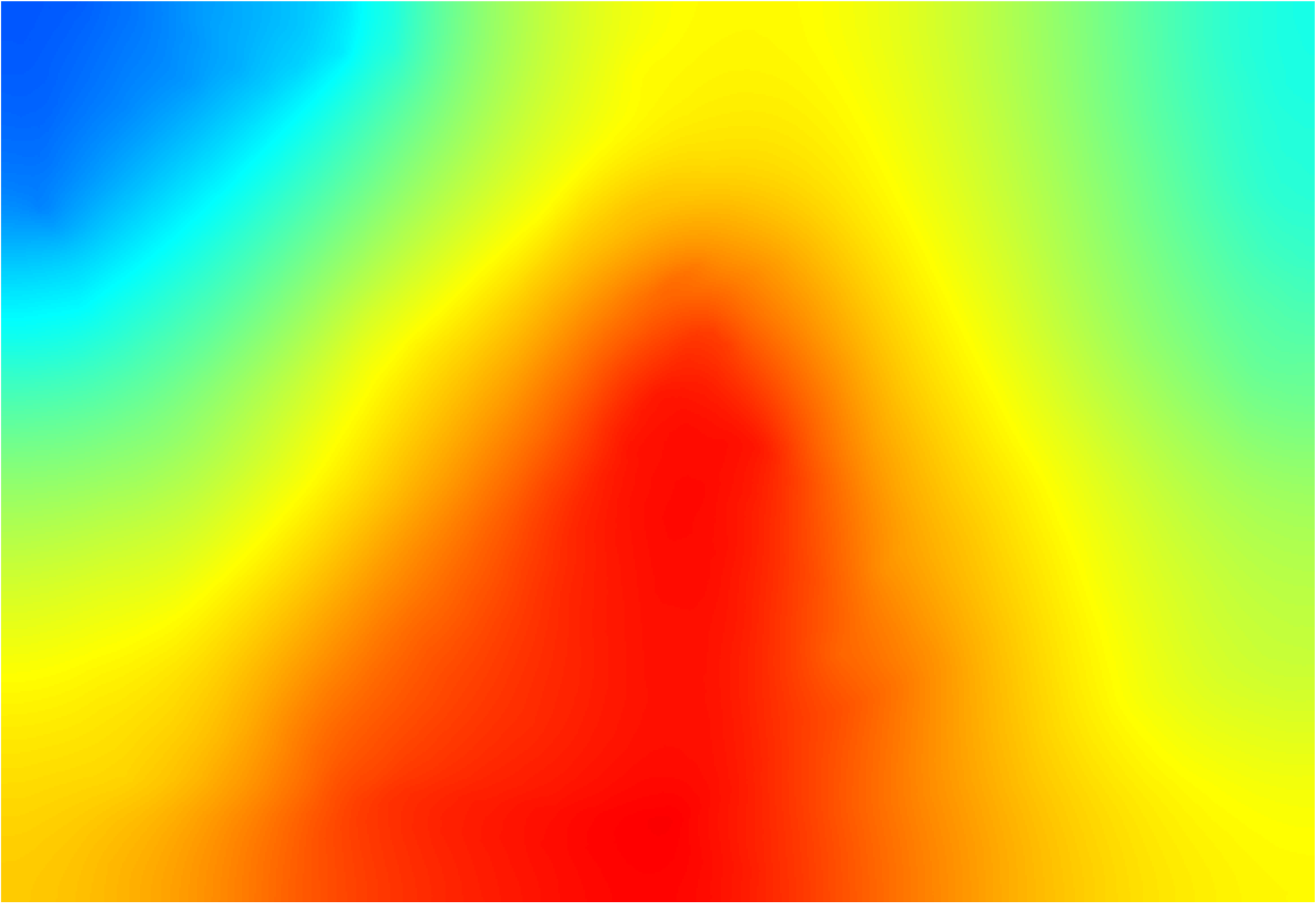}&
    \circimgst{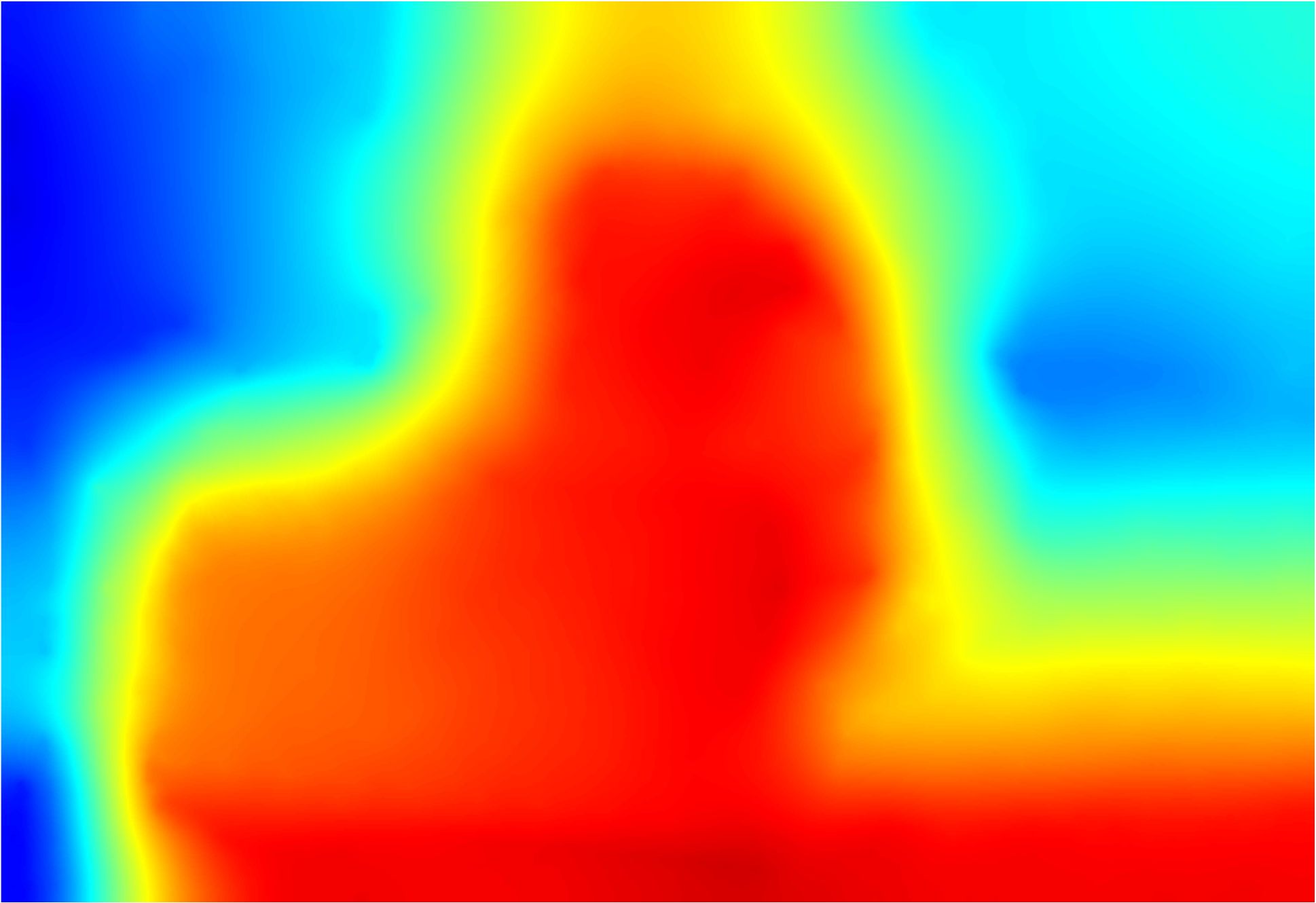}&
    \circimgst{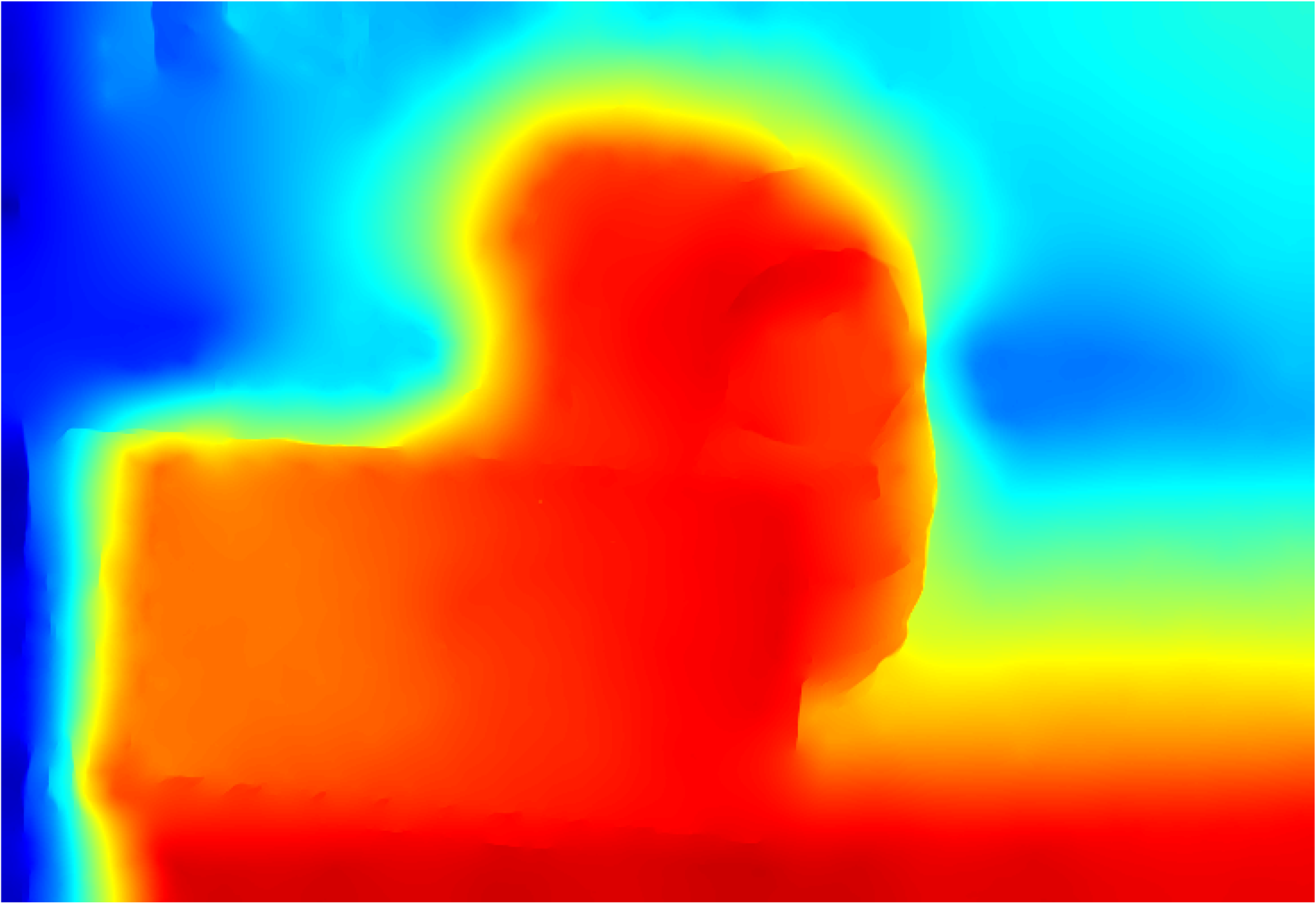}&
    \circimgst{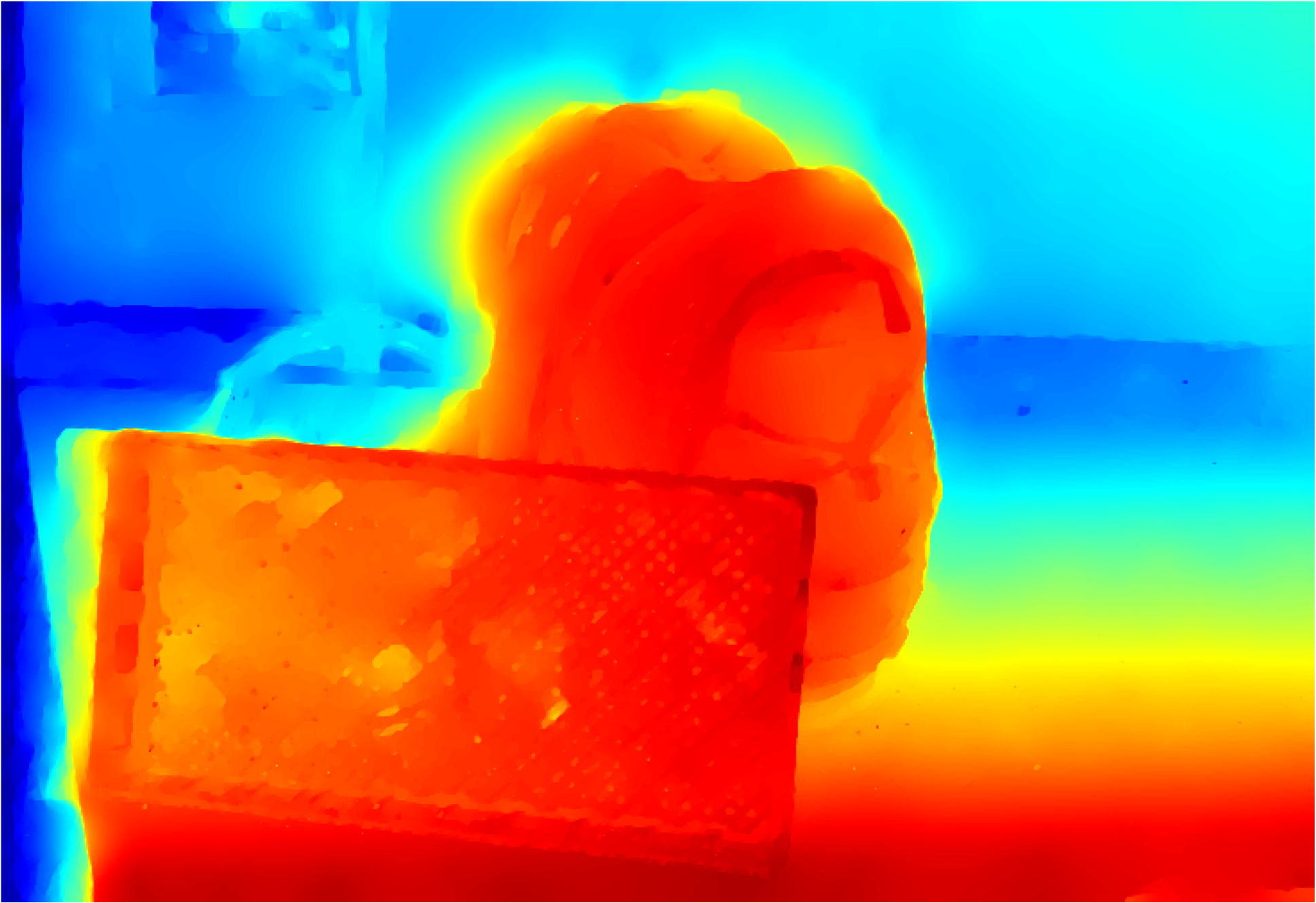}  &
    \circimgst{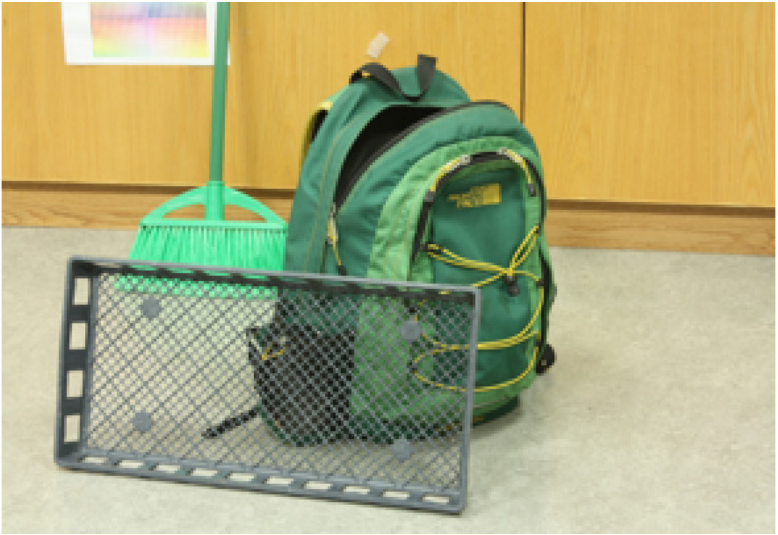} \\
    $\qquad k=1$ & $\qquad k=8$ & $\qquad k=21$ & $\qquad k=60$ & $\qquad \text{Input}$ 
    %\circimgst{./Images/numerical_results/bike/isotropic/none/bike_1}&
    %\circimgst{./Images/numerical_results/bike/isotropic/none/bike_6}&
    %\circimgst{./Images/numerical_results/bike/isotropic/none/bike_11}&
    %\circimgst{./Images/numerical_results/bike/isotropic/none/bike_30} \\
  \end{tabular}%
  \end{center}%
\vspace{-1em}
\caption{\textbf{Lifted Bregman on stereo matching problem with isotropic TV.} TV-\eqref{eq:problem} problem with data term \eqref{eq:stereo_matching_data_term}. This problem is non-convex and non-linear in $u$. At $k=1$ the solution is a coarse approximation of the depth field. As the iteration advances, details are progressively recovered. Although the problem is not of the %convex and positively one-homogeneous 
form OH-\eqref{eq:problem} classically associated with the ISS %inverse scale space 
flow, the results show a qualitative similarity to a nonlinear scale space for this non-convex problem. }
\label{fig:results2}
\end{figure}

\mypar{Conclusion.}
We proposed a combination of the Bregman iteration and a lifting approach with sublabel-accurate discretization in order to extend the Bregman iteration to non-convex energies. If a certain form of the subgradients can be ensured -- which can be shown to be the case with the convex ROF-\eqref{eq:problem} problem and anisotropic TV -- the iterates agree in theory and in practice. In the future, it will be interesting to see if such methods can lead to the development of scale space transformations and nonlinear filters for arbitrary nonconvex data terms. 

\mypar{Acknowledgments.} The authors acknowledge support through DFG grant LE 4064/1-1 “Functional Lifting 2.0: Efficient Convexifications for Imaging and Vision” and NVIDIA Corporation.

\bibliographystyle{abbrv}
\bibliography{InverseScaleSpaceIterationsUsingFunctionalLifting}

\end{document}